\documentclass[12pt]{amsart}
\usepackage{amsfonts}
\usepackage{amscd}
\usepackage{amsmath,amssymb,amsthm}
\usepackage{color}
\usepackage{graphicx}
\usepackage[colorlinks=true,linkcolor=red,citecolor=blue]{hyperref}
\usepackage{subfigure,array,tabularx}

\newtheorem{definition}{Definition}[section]
\newtheorem{remark}{Remark}[section]
\newtheorem{conjecture}{Conjecture}[section]
\newtheorem{theorem}{Theorem}[section]
\newtheorem{lemma}{Lemma}[section]
\newtheorem{corollary}{Corollary}[section]
\numberwithin{equation}{section}

\numberwithin{equation}{section}

\input{tcilatex}

\begin{document}
\title[CR Reilly Formula and Yau Eigenvalue Conjecture ]{On the CR Analogue
of Reilly Formula and Yau Eigenvalue Conjecture}
\author{$^{\ast }$Shu-Cheng Chang$^{1}$}
\address{$^{1}$Department of Mathematics and Taida Institute for
Mathematical Sciences (TIMS), National \ Taiwan University, Taipei 10617,
Taiwan.}
\email{scchang@math.ntu.edu.tw }
\author{$^{\ast }$Chih-Wei Chen$^{2}$}
\address{$^{2}$Department of Mathematics, National Taiwan University, Taipei
10617, Taiwan, R.O.C. }
\email{BabbageTW@gmail.com}
\author{$^{\ast }$Chin-Tung Wu$^{3}$}
\address{$^{3}$Department of applied Mathematics, National Pingtung
University, Pingtung 90003, Taiwan R.O.C.}
\email{ctwu@mail.nptu.edu.tw}
\thanks{$^{\ast }$Research supported in part by the MOST of Taiwan}

\begin{abstract}
In this paper, we derive the CR Reilly's formula and its applications to
studying of the first eigenvalue estimate for CR Dirichlet eigenvalue
problem and embedded $\mathrm{p}$-minimal hypersurfaces. In particular, we
obtain the first Dirichlet eigenvalue estimate in a compact pseudohermitian $%
(2n+1)$-manifold with boundary and the first eigenvalue estimate of the
tangential sublaplacian on closed oriented embedded $\mathrm{p}$-minimal
hypersurfaces in a closed pseudohermitian $(2n+1)$-manifold of vanishing
torsion.
\end{abstract}

\maketitle

\section{Introduction}

In the paper of \cite{Re}, by integral version of Bochner-type formula, R.
Reilly proved so-called Reilly formula which has numerous applications. For
example, Reilly himself applied it to prove a Lichnerowicz type sharp lower
bound for the first eigenvalue of Laplacian on compact Riemannian manifolds
with boundary. In this paper, we will derive the CR version of Reilly's
formula and give some applications as well.

Let $(M,J,\theta )$ be a pseudohermitian $(2n+1)$-manifold (see next section
for basic notions in pseudohermitian geometry). The CR Reilly's formula (\ref%
{CR Reilly's formula}) is involved terms which has no analogue in the
Riemannian case. However, one can relate these extra terms to a third-order
operator $P$ which characterizes CR-pluriharmonic functions (\cite{l1}) and
the fourth-order CR Paneitz operator $P_{0}$ (\cite{gl}).

\begin{definition}
(\cite{gl}, \cite{p}) Let $(M,J,\theta )$ be a pseudohermitian $(2n+1)$%
-manifold. We define 
\begin{equation*}
\begin{array}{c}
P\varphi =\sum_{\gamma ,\beta =1}^{n}(\varphi _{\overline{\gamma }\; \  \beta
}^{\, \overline{\text{ }\gamma }}+inA_{\beta \gamma }\varphi ^{\gamma
})\theta ^{\beta }=\sum_{\beta =1}^{n}(P_{\beta }\varphi )\theta ^{\beta },%
\end{array}%
\end{equation*}%
which is an operator that characterizes CR-pluriharmonic functions. Here 
\begin{equation*}
\begin{array}{c}
P_{\beta }\varphi =\sum_{\gamma =1}^{n}(\varphi _{\overline{\gamma }\; \
\beta }^{\, \text{\ }\overline{\gamma }}+inA_{\beta \gamma }\varphi ^{\gamma
}),\text{ \ }\beta =1,\cdots ,n,%
\end{array}%
\end{equation*}%
and $\overline{P}\varphi =\sum_{\beta =1}^{n}\overline{P}_{\beta }\theta ^{%
\overline{\beta }}$, the conjugate of $P$. The CR Paneitz operator $P_{0}$
is defined by 
\begin{equation}
P_{0}\varphi =4\delta _{b}(P\varphi )+4\overline{\delta }_{b}(\overline{P}%
\varphi ),  \label{Paneitz}
\end{equation}%
where $\delta _{b}$ is the divergence operator that takes $(1,0)$-forms to
functions by $\delta _{b}(\sigma _{\beta }\theta ^{\beta })=\sigma _{\beta
}^{\; \  \beta }$, and similarly, $\overline{\delta }_{b}(\sigma _{\overline{%
\beta }}\theta ^{\overline{\beta }})=\sigma _{\overline{\beta }}^{\; \ 
\overline{\beta }}$.
\end{definition}

We observe that (\cite{gl}) 
\begin{equation}
\begin{array}{lll}
P_{0}\varphi & = & 2\square _{b}\overline{\square }_{b}\varphi -4in(A^{\beta
\gamma }\varphi _{\beta }),_{\gamma } \\ 
& = & 2\overline{\square }_{b}\square _{b}\varphi +4in(A^{\overline{\beta }%
\overline{\gamma }}\varphi _{\overline{\beta }}),_{\overline{\gamma }} \\ 
& = & 2(\Delta _{b}^{2}+n^{2}T^{2})\varphi -4n\func{Re}(iA^{\beta \gamma
}\varphi _{\beta }),_{\gamma }%
\end{array}
\label{111a}
\end{equation}%
for $\square _{b}\varphi =(\overline{\partial }_{b}\overline{\partial ^{\ast
}}_{b}+\overline{\partial ^{\ast }}_{b}\overline{\partial }_{b})\varphi
=(-\Delta _{b}+inT)\varphi =-2\varphi _{\overline{\beta }}^{\, \  \overline{%
\beta }}$.

By using integrating by parts to the CR Bochner formula (\ref{Bochnerformula}%
), we derive the following CR version of Reilly's formula.

\begin{theorem}
\label{Reilly'sformula} Let $(M,J,\theta )$ be a compact pseudohermitian $%
(2n+1)$-manifold with boundary $\Sigma $. Then for any real smooth function $%
\varphi $, we have%
\begin{equation}
\begin{array}{ll}
& \frac{n+1}{n}\int_{M}[(\Delta _{b}\varphi )^{2}-\frac{2n}{n+1}\sum_{\beta
,\gamma }|\varphi _{\beta \gamma }|^{2}]d\mu \\ 
= & \frac{n+2}{4n}\int_{M}\varphi P_{0}\varphi d\mu
+\int_{M}[2Ric-(n+1)Tor]((\nabla _{b}\varphi )_{\mathbb{C}},(\nabla
_{b}\varphi )_{\mathbb{C}})d\mu \\ 
& -\frac{n+2}{2n}iC_{n}\int_{\Sigma }\varphi \left( P_{n}\varphi -P_{%
\overline{n}}\varphi \right) d\Sigma _{p}+\frac{i}{2}C_{n}\int_{\Sigma
}(\varphi ^{\overline{\beta }}B_{n\overline{\beta }}\varphi -\varphi ^{\beta
}B_{\overline{n}\beta }\varphi )d\Sigma _{p} \\ 
& -\frac{1}{4}C_{n}\int_{\Sigma }\varphi _{0}\varphi _{e_{n}}d\Sigma _{p}+%
\frac{3}{4n}C_{n}\int_{\Sigma }\varphi _{e_{2n}}\Delta _{b}\varphi d\Sigma
_{p}+C_{n}\int_{\Sigma }\varphi _{e_{2n}}\Delta _{b}^{t}\varphi d\Sigma _{p}
\\ 
& +\frac{1}{4}C_{n}\int_{\Sigma }H_{p.h}\varphi _{e_{2n}}^{2}d\Sigma _{p}+%
\frac{1}{2}C_{n}\int_{\Sigma }\alpha \varphi _{e_{n}}\varphi
_{e_{2n}}d\Sigma _{p} \\ 
& +\frac{1}{4}C_{n}\int_{\Sigma }\tsum_{j,k=1}^{2n-1}\left \langle \nabla
_{e_{j}}e_{2n},e_{k}\right \rangle \varphi _{e_{j}}\varphi _{e_{k}}d\Sigma
_{p}.%
\end{array}
\label{CR Reilly's formula}
\end{equation}%
Here $P_{0}$ is the CR Paneitz operator on $M.\ C_{n}:=2^{n}n!;$ $B_{\beta 
\overline{\gamma }}\varphi :=\varphi _{\beta \overline{\gamma }}-\frac{1}{n}%
\varphi _{\sigma }{}^{\sigma }h_{\beta \overline{\gamma }}$. $\Delta
_{b}^{t}:=\frac{1}{2}\tsum_{j=1}^{2n-1}[e_{j}{}^{2}-(\nabla
_{e_{j}}e_{j})^{t}]$ is the tangential sublaplacian of $\Sigma $ and $%
H_{p.h} $ is the $p$-mean curvature of $\Sigma $ with respect to the
Legendrian normal $e_{2n},$ $\alpha e_{2n}+T\in T\Sigma $ for some function $%
\alpha $ on $\Sigma \backslash S_{\Sigma },$ the singular set $S_{\Sigma }$
consists of those points where the contact bundle $\xi =\ker \theta $
coincides with the tangent bundle $T\Sigma $ of $\Sigma .$ $(\nabla
_{b}\varphi )_{\mathbb{C}}=\varphi ^{\beta }Z_{\beta }$ is the corresponding
complex $(1,0)$-vector field of $\nabla _{b}\varphi $ and $d\Sigma
_{p}=\theta \wedge e^{1}\wedge e^{n+1}\wedge \cdots \wedge e^{n-1}\wedge
e^{2n-1}\wedge e^{n}$ is the $p$-area element on $\Sigma .$
\end{theorem}

If $(M,J,\theta )$ is a compact pseudohermitian $(2n+1)$-manifold without
boundary, one can check easily that the fourth-order CR Paneitz $P_{0}$ is
self-adjoint. That is 
\begin{equation}
\begin{array}{c}
\int_{M}gP_{0}fd\mu =\int_{M}fP_{0}gd\mu%
\end{array}
\label{1}
\end{equation}%
for all smooth functions $f$ and $g$. However, if $(M,J,\theta )$ is a
compact pseudohermitian $(2n+1)$-manifold with the smooth boundary $\Sigma ,$
it follows from (\ref{A1}) and (\ref{A2}) that (\ref{1}) folds for all
smooth functions with the Dirichlet condition or the Neumann condition as in
(\ref{222}) and (\ref{333}) on $\Sigma $. In particular, it holds in the
situation as in Theorem \ref{TB} and Theorem \ref{TC}.

That is, one can have the following Dirichlet eigenvalue problem or Neumann
eigenvalue problem, respectively : 
\begin{equation}
\left \{ 
\begin{array}{ccl}
P_{0}\varphi & = & \mu _{_{D}}\varphi \  \  \mathrm{on\ }M, \\ 
\varphi & = & 0\  \  \  \  \  \  \mathrm{on\ }\Sigma , \\ 
\Delta _{b}\varphi & = & 0\  \  \  \  \  \  \mathrm{on}\  \Sigma ,%
\end{array}%
\right.  \label{222}
\end{equation}%
and 
\begin{equation}
\left \{ 
\begin{array}{ccl}
P_{0}\phi & = & \mu _{_{N}}\phi \  \  \mathrm{on\ }M, \\ 
\Delta _{b}\phi & = & 0\  \  \  \  \  \  \mathrm{on\ }\Sigma , \\ 
(\Delta _{b}\phi )_{e_{2n}} & = & 0\  \  \  \  \  \  \mathrm{on}\  \Sigma .%
\end{array}%
\right.  \label{333}
\end{equation}%
Hence%
\begin{equation}
\begin{array}{c}
\int_{M}\varphi P_{0}\varphi d\mu \geq \mu _{_{D}}^{1}\int_{M}\varphi
^{2}d\mu%
\end{array}
\label{1a}
\end{equation}%
for the first Dirichlet eigenvalue $\mu _{_{D}}^{1}$ and all smooth
functions with $\varphi =0=\Delta _{b}\varphi $ on $\Sigma .$ Similarly 
\begin{equation}
\begin{array}{c}
\int_{M}\phi P_{0}\phi d\mu \geq \mu _{_{N}}^{1}\int_{M}\phi ^{2}d\mu%
\end{array}
\label{1aa}
\end{equation}%
for the first Neumann eigenvalue $\mu _{_{N}}^{1}$ and all smooth functions
with $\Delta _{b}\phi =0=(\Delta _{b}\phi )_{e_{2n}}$ on $\Sigma $. In
general, $\mu _{_{D}}^{1}$ and $\mu _{_{N}}^{1}$ are not always nonnegative.

\begin{definition}
\label{d1} Let $(M,J,\theta )$ be a compact pseudohermitian $(2n+1)$%
-manifold with the smooth boundary $\Sigma .$ We say that the CR Paneitz
operator $P_{0}$ with respect to $(J,\theta )$ is nonnegative if 
\begin{equation*}
\begin{array}{c}
\int_{M}\varphi P_{0}\varphi d\mu \geq 0%
\end{array}%
\end{equation*}%
for all smooth functions with suitable boundary conditions as in Dirichlet
eigenvalue problem or Neumann eigenvalue problem of $P_{0}$.
\end{definition}

\begin{remark}
\label{r1} Let $(M,J,\theta )$ be a compact pseudohermitian $(2n+1)$%
-manifold of vanishing torsion with the smooth boundary $\Sigma $. It
follows from (\ref{111a}) that the Kohn Laplacian $\square _{b}$ and $%
\overline{\square }$ commute and they are diagonalized simultaneously with 
\begin{equation*}
P_{0}\varphi =2\square _{b}\overline{\square }_{b}\varphi =2\overline{%
\square }_{b}\square _{b}\varphi .
\end{equation*}%
Then the corresponding CR Paneitz operator $P_{0}$ is nonnegative (\cite{ccc}%
). That is 
\begin{equation*}
\mu _{_{D}}^{1}=0=\mu _{_{N}}^{1}.
\end{equation*}

For the first consequence of CR Reilly formula, we can consider the
following Dirichlet eigenvalue problem: 
\begin{equation}
\left \{ 
\begin{array}{ccll}
\Delta _{b}\varphi  & = & -\lambda _{1}\varphi  & \mathrm{on\ }M, \\ 
\varphi  & = & 0 & \mathrm{on\ }\Sigma .%
\end{array}%
\right.   \label{1b}
\end{equation}%
Then we have the following first Dirichlet eigenvalue estimate:
\end{remark}

\begin{theorem}
\label{TB} Let $(M,J,\theta )$ be a compact pseudohermitian $(2n+1)$%
-manifold with the smooth boundary $\Sigma $. If the pseudohermitian mean
curvature $H_{p.h}$ is nonnegative and 
\begin{equation*}
\begin{array}{c}
\lbrack Ric-\frac{n+1}{2}Tor](Z,Z)\geq k\left \langle Z,Z\right \rangle%
\end{array}%
\end{equation*}%
for all $Z\in T_{1,0}$ and a positive constant $k$, then

(i) For $n\geq 2,$ 
\begin{equation*}
\begin{array}{c}
\lambda _{1}\geq \frac{nk}{n+1};%
\end{array}%
\end{equation*}

(ii)\ For $n=1,$ 
\begin{equation*}
\begin{array}{c}
\lambda _{1}\geq \frac{k+\sqrt{k^{2}+6\mu _{_{D}}^{1}}}{4}%
\end{array}%
\end{equation*}%
with $\mu _{_{D}}^{1}\geq -\frac{k^{2}}{6}$. In addition if $P_{0}$ is
nonnegative, in particular if the torsion is vanishing, then%
\begin{equation*}
\begin{array}{c}
\lambda _{1}\geq \frac{k}{2}.%
\end{array}%
\end{equation*}
\end{theorem}

\begin{remark}
It is known that the sharp first eigenvalue estimate is obtained as in \cite%
{gr}, \cite{ll}, \cite{ch}, \cite{cc2} and \cite{fk} in a closed
pseudohermitian $(2n+1)$-manifold.
\end{remark}

Next we can state the second consequence of the CR Reilly formula (\ref{CR
Reilly's formula}) which served as a CR analogue of Yau conjecture \cite{Y}
on the first eigenvalue estimate of embedded oriented minimal hypersurfaces.
We refer to papers of Choi-Wang \cite{cw} and Tang-Yan \cite{ty} which are
related to Yau conjecture.

As before, $\{e_{1},e_{2},\cdots ,e_{n},e_{n+1},\cdots ,e_{2n-1},\alpha
e_{2n}+T\}$ is the base of $T\Sigma $ for some function $\alpha $ on $\Sigma
\backslash S_{\Sigma }$. It follows from (\ref{2014d}) that $\Delta
_{b}^{t}+\alpha e_{n}$ is a self-adjoint operator with respect to the $p$%
-area element $d\Sigma _{p}$ on $\Sigma $. Hence it is natural to consider
the following CR analogue of eigenvalue problem on the embedded closed $p$%
-minimal ($H_{p.h}=0$) hypersurface $\Sigma $ in a closed pseudohermitian $%
(2n+1)$-manifold $(M,J,\theta )$:%
\begin{equation}
L_{\alpha }u=-\lambda _{1}u,  \label{2015}
\end{equation}%
here 
\begin{equation}
\begin{array}{c}
L_{\alpha }:=\Delta _{b}^{t}+\alpha e_{n}.%
\end{array}
\label{2015-1}
\end{equation}

In this paper, we consider the particular case that $\{e_{1},e_{2},\cdots
,e_{n},e_{n+1},\cdots ,e_{2n-1},T\}$ are always tangent to $\Sigma $ ($%
\alpha =0$) as following:%
\begin{equation}
\begin{array}{c}
L_{0}:=\Delta _{b}^{t}.%
\end{array}
\label{2015-2}
\end{equation}%
That is, we have the first eigenvalue estimate of $L_{0}$ on embedded
oriented hypersurfaces of nonnegative pseudohermitian mean curvature:

\begin{theorem}
\label{TC} Let $\Sigma $ be a compact embedded oriented $p$-minimal
hypersurface with $\alpha =0$ in a closed pseudohermitian $(2n+1)$-manifold $%
(M,J,\theta )$ of vanishing torsion. Suppose that the pseudohermitian Ricci
curvature of $M$ is bounded from below by a positive constant $k$. Then

(i) The first non-zero eigenvalue $\lambda _{1}$ of $L_{0}$ on $\Sigma $ has
a lower bound given by 
\begin{equation*}
\begin{array}{c}
\lambda _{1}\geq \frac{k}{2}.%
\end{array}%
\end{equation*}

(ii) In case of $n=1$ if the equality holds, $(M,J,\theta )$ must be a
closed spherical pseudohermitian $3$-manifold and $\Sigma $ be a compact
embedded oriented $p$-minimal surface of genus $\leq 1.$ Moreover, $%
(M,J,\theta )$ is the the standard CR $3$-sphere $(\mathbf{S}^{3},\widehat{J}%
,\widehat{\theta })$ if it is simply connected.
\end{theorem}

Let $(M,J,\theta )$ be a closed spherical pseudohermitian $3$-manifold.
Recall (\cite{cc1}) that we call a CR structure $J$ spherical if Cartan
curvature tensor $Q_{11}$ vanishes identically. Here 
\begin{equation*}
\begin{array}{c}
Q_{11}=\frac{1}{6}W_{11}+\frac{i}{2}WA_{11}-A_{11,0}-\frac{2i}{3}A_{11,%
\overset{\_}{1}1}.%
\end{array}%
\end{equation*}%
Note that $(M,J,\theta )$ is called a spherical pseudohermitian $3$-manifold
if $J$ is a spherical structure. We observe that the spherical structure is
CR invariant and a closed spherical pseudohermitian $3$-manifold $%
(M,J,\theta )$ is locally CR equivalent to the standard pseudohermitian $3$%
-sphere $(\mathbf{S}^{3},\widehat{J},\widehat{\theta }).$

Note that for an $p$-minimal Clifford torus $\Sigma _{0}=S^{1}(\frac{\sqrt{2}%
}{2})\times S^{1}(\frac{\sqrt{2}}{2})\subset \mathbb{R}^{2}\times \mathbb{R}%
^{2}$ in the standard CR $3$-sphere $\mathbf{S}^{3}$ (i.e. $k=2$ and $%
A_{11}=0)$, $T$ is always tangent to $\Sigma _{0}$ (i.e. $\alpha =0$).
Furthermore, the coordinate function $x_{i}$ of $\Sigma _{0}$ is the
eigenfunction of the tangential sublaplacian $\Delta _{b}^{t}$ with 
\begin{equation*}
\Delta _{b}^{t}x_{i}=-x_{i},\  \ i=1,...4.
\end{equation*}

Then in view of Theorem \ref{TC}, we have the following CR analogue of Yau
conjecture on the first eigenvalue estimate of embedded oriented $p$-minimal
surfaces.

\begin{conjecture}
The first eigenvalue of $L_{0}$ on any closed embedded $p$-minimal surface
of genus $\leq 1$ in the standard CR $3$-sphere $(\mathbf{S}^{3},\widehat{J},%
\widehat{\theta })$ is just $1$.
\end{conjecture}

Finally, we propose a CR analogue of Lawson conjecture (\cite{la}):

\begin{conjecture}
Any closed embedded $p$-minimal torus (with $\alpha =0$) in the standard CR $%
3$-sphere $\mathbf{S}^{3}$ is the Clifford torus.
\end{conjecture}

If the Yau conjecture is true for the $2$-torus, it was proved in \cite{mr}
that the Lawson conjecture holds which is to say that the only minimally
embedded torus in $\mathbf{S}^{3}$ is the Clifford torus. However, Lawson
conjecture was solved by S. Brendle \cite{b} recently.

We briefly describe the methods used in our proofs. In section $3$, by using
integrating by parts to the CR Bochner formula (\ref{Bochnerformula}), we
can derive the CR version of Reilly's formula which involving a third order
operator $P$ which characterizes CR-pluriharmonic functions and the CR
Paneitz operator $P_{0}.$ By applying the CR Reilly's formula, we are able
to obtain the first Dirichlet eigenvalue estimate as in section $4$ and
derive the first non-zero eigenvalue estimate of (\ref{2015}) on compact
oriented embedded $p$-minimal hypersurfaces in a closed pseudohermitian $%
(2n+1)$-manifold of vanishing torsion as in section $5$.

\section{Basic Notions in Pseudohermitian Geometry}

We first introduce some basic materials in a pseudohermitian $(2n+1)$%
-manifold. Let $(M,J,\theta )$ be a $(2n+1)$-dimensional, orientable,
contact manifold with contact structure $\xi =\ker \theta $. A CR structure
compatible with $\xi $ is an endomorphism $J:\xi \rightarrow \xi $ such that 
$J^{2}=-1$. We also assume that $J$ satisfies the following integrability
condition: If $X$ and $Y$ are in $\xi $, then so is $[JX,Y]+[X,JY]$ and $%
J([JX,Y]+[X,JY])=[JX,JY]-[X,Y]$. A CR structure $J$ can extend to $\mathbb{C}%
\mathbf{\otimes }\xi $ and decomposes $\mathbb{C}\mathbf{\otimes }\xi $ into
the direct sum of $T_{1,0}$ and $T_{0,1}$ which are eigenspaces of $J$ with
respect to eigenvalues $i$ and $-i$, respectively. A manifold $M$ with a CR
structure is called a CR manifold. A pseudohermitian structure compatible
with $\xi $ is a $CR$ structure $J$ compatible with $\xi $ together with a
choice of contact form $\theta $. Such a choice determines a unique real
vector field $T$ transverse to $\xi $, which is called the characteristic
vector field of $\theta $, such that ${\theta }(T)=1$ and $\mathcal{L}_{T}{%
\theta }=0$ or $d{\theta }(T,{\cdot })=0$. Let $\left \{ T,Z_{\beta },Z_{%
\overline{\beta }}\right \} $ be a frame of $TM\otimes \mathbb{C}$, where $%
Z_{\beta }$ is any local frame of $T_{1,0},\ Z_{\overline{\beta }}=\overline{%
Z_{\beta }}\in T_{0,1}$ and $T$ is the characteristic vector field. Then $%
\left \{ \theta ,\theta ^{\beta },\theta ^{\overline{\beta }}\right \} $,
which is the coframe dual to $\left \{ T,Z_{\beta },Z_{\overline{\beta }%
}\right \} $, satisfies 
\begin{equation}
d\theta =ih_{\beta \overline{\gamma }}\theta ^{\beta }\wedge \theta ^{%
\overline{\gamma }},  \label{dtheta}
\end{equation}%
for some positive definite Hermitian matrix of functions $(h_{\beta 
\overline{\gamma }})$. Actually we can always choose $Z_{\beta }$ such that $%
h_{\beta \overline{\gamma }}=\delta _{\beta \gamma }$; hence, throughout
this note, we assume $h_{\beta \overline{\gamma }}=\delta _{\beta \gamma }$.

The Levi form $\left \langle \ ,\  \right \rangle $ is the Hermitian form on $%
T_{1,0}$ defined by%
\begin{equation*}
\left \langle Z,W\right \rangle =-i\left \langle d\theta ,Z\wedge \overline{W%
}\right \rangle .
\end{equation*}%
We can extend $\left \langle \ ,\  \right \rangle $ to $T_{0,1}$ by defining $%
\left \langle \overline{Z},\overline{W}\right \rangle =\overline{%
\left
\langle Z,W\right \rangle }$ for all $Z,W\in T_{1,0}$. The Levi form
induces naturally a Hermitian form on the dual bundle of $T_{1,0}$, also
denoted by $\left \langle \ ,\  \right \rangle $, and hence on all the
induced tensor bundles. Integrating the Hermitian form (when acting on
sections) over $M$ with respect to the volume form $d\mu =\theta \wedge
(d\theta )^{n}$, we get an inner product on the space of sections of each
tensor bundle.

The pseudohermitian connection of $(J,\theta )$ is the connection $\nabla $
on $TM\otimes \mathbb{C}$ (and extended to tensors) given in terms of a
local frame $Z_{\beta }\in T_{1,0}$ by%
\begin{equation*}
\nabla Z_{\beta }=\theta _{\beta }{}^{\gamma }\otimes Z_{\gamma },\quad
\nabla Z_{\overline{\beta }}=\theta _{\overline{\beta }}{}^{\overline{\gamma 
}}\otimes Z_{\overline{\gamma }},\quad \nabla T=0,
\end{equation*}%
where $\theta _{\beta }{}^{\gamma }$ are the $1$-forms uniquely determined
by the following equations:%
\begin{equation}
\begin{split}
d\theta ^{\beta }& =\theta ^{\gamma }\wedge \theta _{\gamma }{}^{\beta
}+\theta \wedge \tau ^{\beta }, \\
0& =\tau _{\beta }\wedge \theta ^{\beta }, \\
0& =\theta _{\beta }{}^{\gamma }+\theta _{\overline{\gamma }}{}^{\overline{%
\beta }},
\end{split}
\label{structure equs}
\end{equation}%
We can write (by Cartan lemma) $\tau _{\beta }=A_{\beta \gamma }\theta
^{\gamma }$ with $A_{\beta \gamma }=A_{\gamma \beta }$. The curvature of the
Tanaka-Webster connection, expressed in terms of the coframe $\{ \theta
=\theta ^{0},\theta ^{\beta },\theta ^{\overline{\beta }}\}$, is 
\begin{equation*}
\begin{split}
\Pi _{\beta }{}^{\gamma }& =\overline{\Pi _{\bar{\beta}}{}^{\overline{\gamma 
}}}=d\theta _{\beta }{}^{\gamma }-\theta _{\beta }{}^{\sigma }\wedge \theta
_{\sigma }{}^{\gamma }, \\
\Pi _{0}{}^{\beta }& =\Pi _{\beta }{}^{0}=\Pi _{0}{}^{\bar{\beta}}=\Pi _{%
\bar{\beta}}{}^{0}=\Pi _{0}{}^{0}=0.
\end{split}%
\end{equation*}%
Webster showed that $\Pi _{\beta }{}^{\gamma }$ can be written 
\begin{equation*}
\Pi _{\beta }{}^{\gamma }=R_{\beta }{}^{\gamma }{}_{\rho \bar{\sigma}}\theta
^{\rho }\wedge \theta ^{\bar{\sigma}}+W_{\beta }{}^{\gamma }{}_{\rho }\theta
^{\rho }\wedge \theta -W^{\gamma }{}_{\beta \bar{\rho}}\theta ^{\bar{\rho}%
}\wedge \theta +i\theta _{\beta }\wedge \tau ^{\gamma }-i\tau _{\beta
}\wedge \theta ^{\gamma }
\end{equation*}%
where the coefficients satisfy 
\begin{equation*}
R_{\beta \overline{\gamma }\rho \bar{\sigma}}=\overline{R_{\gamma \bar{\beta}%
\sigma \bar{\rho}}}=R_{\overline{\gamma }\beta \bar{\sigma}\rho }=R_{\rho 
\overline{\gamma }\beta \bar{\sigma}},\  \ W_{\beta \overline{\gamma }\rho
}=W_{\rho \overline{\gamma }\beta }.
\end{equation*}

We will denote components of covariant derivatives with indices preceded by
comma; thus write $A_{\rho \beta ,\gamma }$. The indices $\{0,\beta ,%
\overline{\beta }\}$ indicate derivatives with respect to $\{T,Z_{\beta },Z_{%
\overline{\beta }}\}$. For derivatives of a scalar function, we will often
omit the comma, for instance, $u_{\beta }=Z_{\beta }u,\ u_{\gamma \bar{\beta}%
}=Z_{\bar{\beta}}Z_{\gamma }u-\theta _{\gamma }{}^{\rho }(Z_{\bar{\beta}%
})Z_{\rho }u,\ u_{0}=Tu$ for a smooth function $u$ .

For a real function $u$, the subgradient $\nabla _{b}$ is defined by $\nabla
_{b}u\in \xi $ and $\left \langle Z,\nabla _{b}u\right \rangle =du(Z)$ for
all vector fields $Z$ tangent to contact plane. Locally $\nabla
_{b}u=u^{\beta }Z_{\beta }+u^{\overline{\beta }}Z_{\overline{\beta }}$. We
can use the connection to define the subhessian as the complex linear map 
\begin{equation*}
(\nabla ^{H})^{2}u:T_{1,0}\oplus T_{0,1}\rightarrow T_{1,0}\oplus T_{0,1}%
\text{ \ by \ }(\nabla ^{H})^{2}u(Z)=\nabla _{Z}\nabla _{b}u.
\end{equation*}%
In particular, 
\begin{equation*}
\begin{array}{c}
|\nabla _{b}u|^{2}=2\sum_{\beta }u_{\beta }u^{\beta },\quad |\nabla
_{b}^{2}u|^{2}=2\sum_{\beta ,\gamma }(u_{\beta \gamma }u^{\beta \gamma
}+u_{\beta \overline{\gamma }}u^{\beta \overline{\gamma }}).%
\end{array}%
\end{equation*}%
Also the sublaplacian is defined by 
\begin{equation*}
\begin{array}{c}
\Delta _{b}u=Tr\left( (\nabla ^{H})^{2}u\right) =\sum_{\beta }(u_{\beta
}{}^{\beta }+u_{\overline{\beta }}{}^{\overline{\beta }}).%
\end{array}%
\end{equation*}%
The pseudohermitian Ricci tensor and the torsion tensor on $T_{1,0}$ are
defined by 
\begin{equation*}
\begin{array}{l}
Ric(X,Y)=R_{\gamma \bar{\beta}}X^{\gamma }Y^{\bar{\beta}} \\ 
Tor(X,Y)=i\tsum_{\gamma ,\beta }(A_{\overline{\gamma }\bar{\beta}}X^{%
\overline{\gamma }}Y^{\bar{\beta}}-A_{\gamma \beta }X^{\gamma }Y^{\beta }),%
\end{array}%
\end{equation*}%
where $X=X^{\gamma }Z_{\gamma },\ Y=Y^{\beta }Z_{\beta }$.

\section{The CR Reilly's Formula}

Let $M$ be a compact pseudohermitian $(2n+1)$-manifold with boundary $\Sigma 
$. We write $\theta _{\gamma }^{\; \text{\ }\beta }=\omega _{\gamma }^{\;%
\text{\ }\beta }+i\tilde{\omega}_{\gamma }^{\; \text{\ }\beta }$ with $%
\omega _{\gamma }^{\; \text{\ }\beta }=\QTR{up}{\func{Re}}(\theta _{\gamma
}^{\; \text{\ }\beta })$, $\tilde{\omega}_{\gamma }^{\; \text{\ }\beta }=%
\QTR{up}{\func{Im}}(\theta _{\gamma }^{\; \text{\ }\beta })$ and $Z_{\beta }=%
\frac{1}{2}(e_{\beta }-ie_{n+\beta })$ for real vectors $e_{\beta }$, $%
e_{n+\beta }$, $\beta =1,\cdots ,n$. It follows that $e_{n+\beta }=Je_{\beta
}$. Let $e^{\beta }=\QTR{up}{\func{Re}}(\theta ^{\beta })$, $e^{n+\beta }=%
\QTR{up}{\func{Im}}(\theta ^{\beta })$, $\beta =1,\cdots ,n$. Then $\{
\theta ,e^{\beta },e^{n+\beta }\}$ is dual to $\{T,e_{\beta },e_{n+\beta }\}$%
. Now in view of (\ref{dtheta}) and (\ref{structure equs}), we have the
following real version of structure equations: 
\begin{equation*}
\left \{ 
\begin{split}
& d\theta =2\tsum_{\beta }e^{\beta }\wedge e^{n+\beta }, \\
& \nabla e_{\gamma }=\omega _{\gamma }^{\; \text{\ }\beta }\otimes e_{\beta
}+\tilde{\omega}_{\gamma }^{\; \text{\ }\beta }\otimes e_{n+\beta },\text{ \ 
}\nabla e_{n+\gamma }=\omega _{\gamma }^{\; \text{\ }\beta }\otimes
e_{n+\beta }-\tilde{\omega}_{\gamma }^{\; \text{\ }\beta }\otimes e_{\beta },
\\
& de^{\gamma }=e^{\beta }\wedge \omega _{\beta }^{\; \text{\ }\gamma
}-e^{n+\beta }\wedge \tilde{\omega}_{\beta }^{\; \text{\ }\gamma }\text{ 
\textrm{mod} }\theta ;\text{ }de^{n+\gamma }=e^{\beta }\wedge \tilde{\omega}%
_{\beta }^{\text{ \ }\gamma }+e^{n+\beta }\wedge \omega _{\beta }^{\text{ \ }%
\gamma }\text{ \textrm{mod} }\theta .
\end{split}%
\right.
\end{equation*}

Let $\Sigma $ be a surface contained in $M$. The singular set $S_{\Sigma }$
consists of those points where $\xi $ coincides with the tangent bundle $%
T\Sigma $ of $\Sigma $. It is easy to see that $S_{\Sigma }$ is a closed
set. On $\xi ,$ we can associate a natural metric $\langle $ $,$ $\rangle
_{G}=\frac{1}{2}d\theta (\cdot ,J\cdot )$ call the Levi metric. For a vector 
$v\in \xi ,$ we define the length of $v$ by $\left \vert v\right \vert
_{G}^{2}=\langle v,v\rangle _{G}.$ With respect to the Levi metric, we can
take unit vector fields $e_{1},\cdots ,e_{2n-1}\in \xi \cap T\Sigma $ on $%
\Sigma \backslash S_{\Sigma }$, called the characteristic fields and $%
e_{2n}=Je_{n}$, called the Legendrian normal. The $p$(pseudohermitian)-mean
curvature $H_{p.h}$ on $\Sigma \backslash S_{\Sigma }$ is defined by 
\begin{equation*}
\begin{array}{c}
H_{p.h}=-\sum_{j=1}^{2n-1}\left \langle \nabla _{e_{j}}e_{2n},e_{j}\right
\rangle .%
\end{array}%
\end{equation*}%
For $e_{1},\cdots ,e_{2n-1}$ being characteristic fields, we have the $p$%
-area element 
\begin{equation*}
d\Sigma _{p}=\theta \wedge e^{1}\wedge e^{n+1}\wedge \cdots \wedge
e^{n-1}\wedge e^{2n-1}\wedge e^{n}
\end{equation*}%
on $\Sigma $ and all surface integrals over $\Sigma $ are with respect to
this $2n$-form $d\Sigma _{p}$. Note that $d\Sigma _{p}$ continuously extends
over the singular set $S_{\Sigma }$ and vanishes on $S_{\Sigma }$.

We also write $\varphi _{e_{j}}=e_{j}\varphi $ and $\nabla _{b}\varphi =%
\frac{1}{2}(\varphi _{e_{\beta }}e_{\beta }+\varphi _{e_{n+\beta
}}e_{n+\beta })$. Moreover, $\varphi _{e_{j}e_{k}}=e_{k}e_{j}\varphi -\nabla
_{e_{k}}e_{j}\varphi $ and $\Delta _{b}\varphi =\frac{1}{2}\tsum_{\beta
}(\varphi _{e_{\beta }e_{\beta }}+\varphi _{e_{n+\beta }e_{n+\beta }})$.
Next we define the subdivergence operator $div_{b}(\cdot )$ by $%
div_{b}(W)=W^{\beta },_{\beta }+W^{\overline{\beta }},_{\overline{\beta }}$
for all vector fields $W=W^{\beta }Z_{\beta }+W^{\overline{\beta }}Z_{%
\overline{\beta }}$ and its real version is $div_{b}(W)=\varphi _{\beta
,e_{\beta }}+\psi _{n+\beta ,e_{n+\beta }}$ for $W=\varphi _{\beta }e_{\beta
}+\psi _{n+\beta }e_{n+\beta }$. We define the tangential subgradient $%
\nabla _{b}^{t}$ of a function $\varphi $ by $\nabla _{b}^{t}\varphi =\nabla
_{b}\varphi -\langle \nabla _{b}\varphi ,e_{2n}\rangle _{G}e_{2n}$ and the
tangent sublaplacian $\Delta _{b}^{t}$ of $\varphi $ by $\Delta
_{b}^{t}\varphi =\frac{1}{2}\sum_{j=1}^{2n-1}[(e_{j})^{2}\varphi -(\nabla
_{e_{j}}e_{j})^{t}\varphi ],$ where $(\nabla _{e_{j}}e_{j})^{t}$ is the
tangential part of $\nabla _{e_{j}}e_{j}$.

We first recall the following CR Bochner formula.

\begin{lemma}
Let $(M,J,\theta )$ be a pseudohermitian $(2n+1)$-manifold. For a real
function $\varphi $, we have%
\begin{equation}
\begin{array}{lll}
\frac{1}{2}\Delta _{b}|\nabla _{b}\varphi |^{2} & = & |\nabla
_{b}^{2}\varphi |^{2}+\langle \nabla _{b}\varphi ,\nabla _{b}\Delta
_{b}\varphi \rangle \\ 
&  & +[2Ric-(n-2)Tor]((\nabla _{b}\varphi )_{\mathbb{C}},(\nabla _{b}\varphi
)_{\mathbb{C}})+2\langle J\nabla _{b}\varphi ,\nabla _{b}\varphi _{0}\rangle
,%
\end{array}
\label{Bochnerformula}
\end{equation}%
where $(\nabla _{b}\varphi )_{\mathbb{C}}=\varphi ^{\beta }Z_{\beta }$ is
the corresponding complex $(1,0)$-vector field of $\nabla _{b}\varphi $.
\end{lemma}

The proof of the above formula follows from the Bochner formula (Lemma 3 in 
\cite{gr}) derived by A. Greenleaf and using the commutation relation (see
Lemma 2.2 in \cite{cc1}) 
\begin{equation*}
\begin{array}{c}
i\tsum_{\beta }(\varphi _{\beta }\varphi _{\overline{\beta }0}-\varphi _{%
\overline{\beta }}\varphi _{\beta 0})=i\tsum_{\beta }(\varphi _{\beta
}\varphi _{0\overline{\beta }}-\varphi _{\overline{\beta }}\varphi _{0\beta
})-Tor((\nabla _{b}\varphi )_{\mathbb{C}},(\nabla _{b}\varphi )_{\mathbb{C}%
}).%
\end{array}%
\end{equation*}%
From \cite{cc1}, we can relate $\langle J\nabla _{b}\varphi ,\nabla
_{b}\varphi _{0}\rangle $ with $\langle \nabla _{b}\varphi ,\nabla
_{b}\Delta _{b}\varphi \rangle $ by%
\begin{equation}
\begin{array}{c}
\langle J\nabla _{b}\varphi ,\nabla _{b}\varphi _{0}\rangle =\frac{1}{n}%
\langle \nabla _{b}\varphi ,\nabla _{b}\Delta _{b}\varphi \rangle
-2Tor((\nabla _{b}\varphi )_{\mathbb{C}},(\nabla _{b}\varphi )_{\mathbb{C}})-%
\frac{2}{n}\langle P\varphi +\overline{P}\varphi ,d_{b}\varphi \rangle .%
\end{array}
\label{relation formula}
\end{equation}

For the proof of Reilly's formula, we first need a series of formulae. In
particular, one derives the following CR version of divergence theorem and
Green's identity for a compact pseudohermitian $(2n+1)$-manifold $M$ with
boundary $\Sigma $. Note that $d\Sigma _{p}$ vanishes on $S_{\Sigma }$.

\begin{lemma}
\label{DivergenceTheorem} (Divergence Theorem) Let $(M,J,\theta )$ be a
compact pseudohermitian $(2n+1)$-manifold with boundary $\Sigma .$ For a
real function $\varphi $, we have 
\begin{equation}
\begin{array}{c}
\int_{M}\Delta _{b}\varphi d\mu =\int_{M}\QTR{up}{div}_{b}(\nabla
_{b}\varphi )d\mu =\frac{1}{2}C_{n}\int_{\Sigma }\varphi _{e_{2n}}d\Sigma
_{p}=C_{n}\int_{\Sigma }\langle \nabla _{b}\varphi ,e_{2n}\rangle
_{G}d\Sigma _{p},%
\end{array}
\label{C}
\end{equation}%
\begin{equation}
\begin{array}{c}
\int_{M}\varphi \varphi _{00}d\mu +\int_{M}\varphi _{0}^{2}d\mu
=-C_{n}\int_{\Sigma }\alpha \varphi \varphi _{0}d\Sigma _{p}.%
\end{array}
\label{C1}
\end{equation}%
Here $d\Sigma _{p}=\theta \wedge e^{1}\wedge e^{n+1}\wedge \cdots \wedge
e^{n-1}\wedge e^{2n-1}\wedge e^{n}$ is the $p$-area element of $\Sigma $ and 
$C_{n}=2^{n}n!$.
\end{lemma}

\begin{proof}
By the Stoke's theorem, we have%
\begin{equation*}
\begin{array}{lll}
\int_{M}\Delta _{b}\varphi d\mu & = & \frac{1}{2}\int_{M}\tsum_{\beta
}(\varphi _{e_{\beta }e_{\beta }}+\varphi _{e_{n+\beta }e_{n+\beta
}})2^{n}n!\theta \wedge e^{1}\wedge e^{n+1}\wedge \cdots \wedge e^{n}\wedge
e^{2n} \\ 
& = & 2^{n-1}n!\int_{M}\tsum_{\beta }d[-\varphi _{e_{\beta }}\theta \wedge
e^{1}\wedge e^{n+1}\wedge \cdots \wedge \widehat{e^{\beta }}\wedge
e^{n+\beta }\wedge \cdots \wedge e^{n}\wedge e^{2n} \\ 
&  & \text{ \  \  \  \  \  \  \  \  \  \  \  \  \  \  \  \  \  \ }+\varphi _{e_{n+\beta
}}\theta \wedge e^{1}\wedge e^{n+1}\wedge \cdots \wedge e^{\beta }\wedge 
\widehat{e^{n+\beta }}\wedge \cdots \wedge e^{n}\wedge e^{2n}] \\ 
& = & 2^{n-1}n!\int_{\Sigma }\varphi _{e_{2n}}\theta \wedge e^{1}\wedge
e^{n+1}\wedge \cdots \wedge e^{n-1}\wedge e^{2n-1}\wedge e^{n} \\ 
& = & C_{n}\int_{\Sigma }\langle \nabla _{b}\varphi ,e_{2n}\rangle
_{G}d\Sigma _{p}.%
\end{array}%
\end{equation*}%
Here we used $d\mu =\theta \wedge (d\theta )^{n}=C_{n}\theta \wedge
e^{1}\wedge e^{n+1}\wedge \cdots \wedge e^{n}\wedge e^{2n}$ and the fact
that the $2n$-forms $\theta \wedge e^{1}\wedge e^{n+1}\wedge \cdots \wedge 
\widehat{e^{\beta }}\wedge e^{n+\beta }\wedge \cdots \wedge e^{n}\wedge
e^{2n}$ vanish on $S_{\Sigma }$ for $\beta =1,\cdots ,n$ and so are $\theta
\wedge e^{1}\wedge e^{n+1}\wedge \cdots \wedge e^{\beta }\wedge \widehat{%
e^{n+\beta }}\wedge \cdots \wedge e^{n}\wedge e^{2n}$ for $\beta =1,\cdots
,n-1$, since $e_{j}$ are tangent along $\Sigma $ for $j=1,\cdots ,2n-1$.

The second equation follows easily from Stoke's theorem as above 
\begin{equation*}
\begin{array}{lll}
\int_{M}\varphi \varphi _{00}d\mu +\int_{M}\varphi _{0}^{2}d\mu & = & 
C_{n}\int_{M}d(\varphi \varphi _{0}e^{1}\wedge e^{n+1}\wedge \cdots \wedge
e^{n}\wedge e^{2n}) \\ 
& = & C_{n}\int_{\Sigma }\varphi \varphi _{0}e^{1}\wedge e^{n+1}\wedge
\cdots \wedge e^{n}\wedge e^{2n}%
\end{array}%
\end{equation*}%
and the help of the identity $e^{2n}\wedge e^{n}=\alpha \theta \wedge e^{n}$
on $\Sigma \backslash S_{\Sigma }.$
\end{proof}

\begin{corollary}
\label{Green's identity} (Green's identity) Let $(M,J,\theta )$ be a compact
pseudohermitian $(2n+1)$-manifold with boundary $\Sigma .$ For real
functions $\varphi $ and $\psi $, 
\begin{equation}
\begin{array}{c}
\int_{M}\psi \Delta _{b}\varphi d\mu +\int_{M}\langle \nabla _{b}\varphi
,\nabla _{b}\psi \rangle d\mu =\frac{1}{2}C_{n}\int_{\Sigma }\psi \varphi
_{e_{2n}}d\Sigma _{p}.%
\end{array}
\label{B}
\end{equation}
\end{corollary}

\begin{proof}
It is easy to check that $\QTR{up}{div}_{b}(\psi \nabla _{b}\varphi )=\psi
\Delta _{b}\varphi +\langle \nabla _{b}\varphi ,\nabla _{b}\psi \rangle $
and then the result follows from the CR version of divergence theorem.
\end{proof}

\begin{lemma}
\label{Jvarphi0} Let $(M,J,\theta )$ be a compact pseudohermitian $(2n+1)$%
-manifold with boundary $\Sigma .$ For any real smooth function $\varphi $, 
\begin{equation}
\begin{array}{c}
\int_{M}\langle J\nabla _{b}\varphi ,\nabla _{b}\varphi _{0}\rangle d\mu
+n\int_{M}\varphi _{0}^{2}d\mu =\frac{1}{2}C_{n}\int_{\Sigma }\varphi
_{0}\varphi _{e_{n}}d\Sigma _{p}.%
\end{array}
\label{eqJvarphi0}
\end{equation}
\end{lemma}

\begin{proof}
Since $\QTR{up}{div}_{b}(\left( J\nabla _{b}\varphi \right) \varphi
_{0})=\langle J\nabla _{b}\varphi ,\nabla _{b}\varphi _{0}\rangle +n\varphi
_{0}^{2}$ and by the divergence theorem (\ref{C}), we have%
\begin{equation*}
\begin{array}{ll}
& \int_{M}\langle J\nabla _{b}\varphi ,\nabla _{b}\varphi _{0}\rangle d\mu
+n\int_{M}\varphi _{0}^{2}d\mu \\ 
= & \int_{M}\QTR{up}{div}_{b}(\left( J\nabla _{b}\varphi \right) \varphi
_{0})d\mu =C_{n}\int_{\Sigma }\langle \left( J\nabla _{b}\varphi \right)
\varphi _{0},e_{2n}\rangle _{G}d\Sigma _{p}=\frac{1}{2}C_{n}\int_{\Sigma
}\varphi _{0}\varphi _{e_{n}}d\Sigma _{p}.%
\end{array}%
\end{equation*}
\end{proof}

\begin{lemma}
\label{P&Paneitz} Let $(M,J,\theta )$ be a compact pseudohermitian $(2n+1)$%
-manifold with boundary $\Sigma .$ For any real smooth function $\varphi $, 
\begin{equation}
\begin{array}{c}
\int_{M}\langle P\varphi +\overline{P}\varphi ,d_{b}\varphi \rangle d\mu +%
\frac{1}{4}\int_{M}(P_{0}\varphi )\varphi d\mu =\frac{1}{2}%
iC_{n}\int_{\Sigma }\varphi \left( P_{n}\varphi -P_{\overline{n}}\varphi
\right) d\Sigma _{p}.%
\end{array}
\label{eqP&Paneitz}
\end{equation}
\end{lemma}

\begin{proof}
It can be easily checked that 
\begin{equation*}
\begin{array}{c}
\QTR{up}{div}_{b}\left( (\varphi P^{\beta }\varphi )Z_{\beta }+(\varphi P^{%
\overline{\beta }}\varphi )Z_{\overline{\beta }}\right) =\langle P\varphi +%
\overline{P}\varphi ,d_{b}\varphi \rangle +\frac{1}{4}\varphi P_{0}\varphi .%
\end{array}%
\end{equation*}%
We then have by the divergence theorem (\ref{C}) 
\begin{equation*}
\begin{array}{ll}
& \int_{M}\langle P\varphi +\overline{P}\varphi ,d_{b}\varphi \rangle d\mu +%
\frac{1}{4}\int_{M}(P_{0}\varphi )\varphi d\mu \\ 
= & C_{n}\int_{\Sigma }\left \langle (\varphi P^{\beta }\varphi )Z_{\beta
}+(\varphi P^{\overline{\beta }}\varphi )Z_{\overline{\beta }},e_{2n}\right
\rangle _{G}d\Sigma _{p}=\frac{1}{2}iC_{n}\int_{\Sigma }\varphi \left(
P_{n}\varphi -P_{\overline{n}}\varphi \right) d\Sigma _{p}.%
\end{array}%
\end{equation*}
\end{proof}

\begin{lemma}
Let $(M,J,\theta )$ be a compact pseudohermitian $(2n+1)$-manifold with
boundary $\Sigma $. For real-valued functions $\varphi $ on $\Sigma ,$%
\begin{equation}
\begin{array}{c}
\int_{\Sigma }\left( \varphi _{e_{n}}+2\alpha \varphi \right) d\Sigma _{p}=0;%
\end{array}
\label{2014a}
\end{equation}%
\begin{equation}
\begin{array}{c}
\int_{\Sigma }[\varphi _{\overline{\beta }}+(\tsum_{\gamma \neq n}\theta _{%
\overline{\beta }}^{\; \text{\ }\overline{\gamma }}(Z_{\overline{\gamma }})+%
\frac{1}{2}\theta _{\overline{\beta }}^{\; \text{\ }\overline{n}%
}(e_{n}))\varphi ]d\Sigma _{p}=0\text{ for any }\beta \neq n;%
\end{array}
\label{2014b}
\end{equation}%
\begin{equation}
\begin{array}{c}
\int_{\Sigma }[\varphi _{0}+\alpha \varphi _{e_{2n}}-(\alpha \tilde{\omega}%
_{n}^{\; \text{\ }n}(e_{n})-\func{Re}A_{\overline{n}\overline{n}})\varphi
]d\Sigma _{p}=0.%
\end{array}
\label{2014c}
\end{equation}
\end{lemma}

\begin{proof}
By the Stoke's theorem, we have 
\begin{equation*}
\begin{array}{lll}
\frac{1}{2}C_{n}\int_{\Sigma }\varphi _{e_{n}}d\Sigma _{p} & = & 
\int_{\Sigma }\varphi _{e_{n}}\theta \wedge \left( d\theta \right)
^{n-1}\wedge e^{n} \\ 
& = & -\int_{\Sigma }d\varphi \wedge \theta \wedge \left( d\theta \right)
^{n-1}+\int_{\Sigma }\varphi _{e_{2n}}e^{2n}\wedge \theta \wedge \left(
d\theta \right) ^{n-1} \\ 
& = & -\int_{\Sigma }d(\varphi \theta \wedge \left( d\theta \right)
^{n-1})+\int_{\Sigma }\varphi d\theta \wedge \left( d\theta \right) ^{n-1}
\\ 
& = & \int_{\Sigma }2\varphi e^{n}\wedge e^{2n}\wedge \left( d\theta \right)
^{n-1}=-\int_{\Sigma }2\alpha \varphi \theta \wedge e^{n}\wedge \left(
d\theta \right) ^{n-1} \\ 
& = & -C_{n}\int_{\Sigma }\alpha \varphi d\Sigma _{p},%
\end{array}%
\end{equation*}%
where we used the identities $\theta \wedge \left( d\theta \right)
^{n-1}\wedge e^{2n}=0$ on $\Sigma $ since $e_{n}$ is tangent along $\Sigma ,$
$d\theta =2\tsum_{\beta =1}^{n}e^{\beta }\wedge e^{n+\beta }$ and $%
e^{2n}\wedge e^{n}=\alpha \theta \wedge e^{n}$ on $\Sigma \backslash
S_{\Sigma }.$

For the second equation, we compute 
\begin{equation*}
\begin{array}{ll}
& \int_{\Sigma }\varphi _{\overline{\beta }}\theta \wedge \left( d\theta
\right) ^{n-1}\wedge e^{n}=\int_{\Sigma }\varphi _{\overline{\beta }}\theta
\wedge \theta ^{\beta }\wedge \theta ^{\overline{\beta }}\wedge
(\tsum_{j=1}^{n-1}\underset{j\neq \beta }{\wedge }\theta ^{j}\wedge \theta ^{%
\overline{j}})\wedge e^{n} \\ 
= & \int_{\Sigma }d\varphi \wedge \theta \wedge \theta ^{\beta }\wedge
(\tsum_{j=1}^{n-1}\underset{j\neq \beta }{\wedge }\theta ^{j}\wedge \theta ^{%
\overline{j}})\wedge e^{n}=-\int_{\Sigma }\varphi d[\theta \wedge \theta
^{\beta }\wedge (\left( d\theta \right) ^{n-2})\wedge e^{n}] \\ 
= & \int_{\Sigma }\varphi \lbrack \theta \wedge d\theta ^{\beta }\wedge
(\left( d\theta \right) ^{n-2})\wedge e^{n}]-\int_{\Sigma }\varphi \lbrack
\theta \wedge \theta ^{\beta }\wedge (\left( d\theta \right) ^{n-2})\wedge
de^{n}] \\ 
= & \int_{\Sigma }\varphi \lbrack \theta \wedge (\theta ^{\gamma }\wedge
\theta _{\gamma }{}^{\beta }+\theta \wedge \tau ^{\beta })\wedge
(\tsum_{j=1}^{n-1}\underset{j\neq \gamma }{\wedge }\theta ^{j}\wedge \theta
^{\overline{j}})\wedge e^{n}] \\ 
& -\int_{\Sigma }\frac{1}{2}\varphi \lbrack \theta \wedge \theta ^{\beta
}\wedge (\tsum_{j=1}^{n-1}\underset{j\neq \beta }{\wedge }\theta ^{j}\wedge
\theta ^{\overline{j}})\wedge (\tsum_{\gamma \neq n}\theta _{\overline{%
\gamma }}{}^{\overline{n}}(e_{n})\theta ^{\overline{\gamma }})\wedge e^{n}]
\\ 
= & \int_{\Sigma }\left( \tsum_{\gamma \neq n}\theta _{\gamma }{}^{\beta
}(Z_{\overline{\gamma }})-\frac{1}{2}\theta _{\overline{\beta }}^{\;%
\overline{n}}(e_{n})\right) \varphi \theta \wedge \theta ^{\beta }\wedge
\theta ^{\overline{\beta }}\wedge (\tsum_{j=1}^{n-1}\underset{j\neq \beta }{%
\wedge }\theta ^{j}\wedge \theta ^{\overline{j}})\wedge e^{n} \\ 
= & -\int_{\Sigma }\left( \tsum_{\gamma \neq n}\theta _{\overline{\beta }%
}^{\; \overline{\gamma }}(Z_{\overline{\gamma }})+\frac{1}{2}\theta _{%
\overline{\beta }}^{\; \overline{n}}(e_{n})\right) \varphi \theta \wedge
\left( d\theta \right) ^{n-1}\wedge e^{n},%
\end{array}%
\end{equation*}%
where we used $de^{n}=\frac{1}{2}(\theta ^{\gamma }\wedge \theta _{\gamma
}{}^{n}+\theta ^{\overline{\gamma }}\wedge \theta _{\overline{\gamma }}{}^{%
\overline{n}})=\frac{1}{2}\tsum_{\gamma \neq n}\theta _{\overline{\gamma }%
}{}^{\overline{n}}(e_{n})\theta ^{\overline{\gamma }}\wedge e^{n}$ $\mathrm{%
mod}$ $\theta ,$ $e^{2n}$\ on $\Sigma .$

The same compute for the third equation yields%
\begin{equation*}
\begin{array}{ll}
& \int_{\Sigma }\varphi _{0}\theta \wedge \left( d\theta \right)
^{n-1}\wedge e^{n} \\ 
= & \int_{\Sigma }d\varphi \wedge \left( d\theta \right) ^{n-1}\wedge
e^{n}-\int_{\Sigma }\varphi _{e_{2n}}e^{2n}\wedge e^{n}\wedge \left( d\theta
\right) ^{n-1} \\ 
= & \int_{\Sigma }d(\varphi \left( d\theta \right) ^{n-1}\wedge
e^{n})-\int_{\Sigma }\varphi \left( d\theta \right) ^{n-1}\wedge
de^{n}-\int_{\Sigma }\alpha \varphi _{e_{2n}}\theta \wedge \left( d\theta
\right) ^{n-1}\wedge e^{n} \\ 
= & \int_{\Sigma }\varphi \left( d\theta \right) ^{n-1}\wedge \lbrack \tilde{%
\omega}_{n}^{\;n}(e_{n})e^{2n}\wedge e^{n}-\func{Re}A_{\overline{n}\overline{%
n}}\theta \wedge e^{n}]-\int_{\Sigma }\alpha \varphi _{e_{2n}}\theta \wedge
\left( d\theta \right) ^{n-1}\wedge e^{n} \\ 
= & \int_{\Sigma }[(\alpha \tilde{\omega}_{n}^{\;n}(e_{n})-\func{Re}A_{%
\overline{n}\overline{n}})\varphi -\alpha \varphi _{e_{2n}}]\theta \wedge
\left( d\theta \right) ^{n-1}\wedge e^{n}.%
\end{array}%
\end{equation*}
\end{proof}

\begin{lemma}
Let $(M,J,\theta )$ be a compact pseudohermitian $(2n+1)$-manifold with
boundary $\Sigma $. For real-valued functions $\varphi $ and $\psi $ on $%
\Sigma ,$ we have%
\begin{equation}
\begin{array}{c}
\int_{\Sigma }\psi (\Delta _{b}^{t}+\alpha e_{n})\varphi d\Sigma
_{p}=\int_{\Sigma }\varphi (\Delta _{b}^{t}+\alpha e_{n})\psi d\Sigma _{p}.%
\end{array}
\label{2014d}
\end{equation}
\end{lemma}

This Lemma implies that $\Delta _{b}^{t}+\alpha e_{n}$ is a self-adjoint
operator with respect to the $p$-area element $d\Sigma _{p}$ on $\Sigma .$

\textbf{The Proof of Theorem}\textup{\textbf{\  \ref{Reilly'sformula}:}}

\begin{proof}
By integrating the CR version of Bochner formula (\ref{Bochnerformula}), we
have%
\begin{equation*}
\begin{array}{lll}
\frac{1}{2}\int_{M}\Delta _{b}|\nabla _{b}\varphi |^{2}d\mu & = & 
\int_{M}|\nabla _{b}^{2}\varphi |^{2}d\mu +\int_{M}\langle \nabla
_{b}\varphi ,\nabla _{b}\Delta _{b}\varphi \rangle d\mu \\ 
&  & +\int_{M}[2Ric-(n-2)Tor]((\nabla _{b}\varphi )_{\mathbb{C}},(\nabla
_{b}\varphi )_{\mathbb{C}})d\mu \\ 
&  & +2\int_{M}\langle J\nabla _{b}\varphi ,\nabla _{b}\varphi _{0}\rangle
d\mu .%
\end{array}%
\end{equation*}%
Note that 
\begin{equation*}
\begin{array}{c}
\tsum_{\beta ,\gamma }|\varphi _{\beta \overline{\gamma }}|^{2}=\tsum_{\beta
,\gamma }|\varphi _{\beta \overline{\gamma }}-\frac{1}{n}\varphi _{\sigma
}{}^{\sigma }h_{\beta \overline{\gamma }}|^{2}+\frac{1}{4n}\left( \Delta
_{b}\varphi \right) ^{2}+\frac{n}{4}\varphi _{0}^{2}.%
\end{array}%
\end{equation*}%
It follows from the CR Green's identity (\ref{B}) with $\psi =\Delta
_{b}\varphi $ and (\ref{eqJvarphi0}), that 
\begin{equation}
\begin{array}{ll}
& \frac{1}{2}\int_{M}\Delta _{b}|\nabla _{b}\varphi |^{2}d\mu \\ 
= & 2\int_{M}\tsum_{\beta ,\gamma }|\varphi _{\beta \gamma }|^{2}d\mu
+2\int_{M}\tsum_{\gamma ,\beta }|\varphi _{\beta \overline{\gamma }}-\frac{1%
}{n}\varphi _{\sigma }{}^{\sigma }h_{\beta \overline{\gamma }}|^{2}d\mu \\ 
& -\frac{3n}{2}\int_{M}\varphi _{0}^{2}d\mu +C_{n}\int_{\Sigma }\varphi
_{0}\varphi _{e_{n}}d\Sigma _{p}+\frac{1}{2}C_{n}\int_{\Sigma }(\Delta
_{b}\varphi )\varphi _{e_{2n}}d\Sigma _{p} \\ 
& -\frac{2n-1}{2n}\int_{M}(\Delta _{b}\varphi )^{2}d\mu
+\int_{M}[2Ric-(n-2)Tor]((\nabla _{b}\varphi )_{\mathbb{C}},(\nabla
_{b}\varphi )_{\mathbb{C}}).%
\end{array}
\label{temporary05}
\end{equation}%
By combining (\ref{eqJvarphi0}), (\ref{relation formula}), (\ref{B}) and (%
\ref{eqP&Paneitz}), we have%
\begin{equation}
\begin{array}{lll}
n\int_{M}\varphi _{0}^{2}d\mu & = & \frac{1}{n}\int_{M}(\Delta _{b}\varphi
)^{2}d\mu -\frac{1}{2n}C_{n}\int_{\Sigma }(\Delta _{b}\varphi )\varphi
_{e_{2n}}d\Sigma _{p} \\ 
&  & -\frac{1}{2n}\int_{M}\varphi P_{0}\varphi d\mu +\frac{1}{n}%
iC_{n}\int_{\Sigma }\varphi \left( P_{n}\varphi -P_{\overline{n}}\varphi
\right) d\Sigma _{p} \\ 
&  & +\frac{1}{2}C_{n}\int_{\Sigma }\varphi _{0}\varphi _{e_{n}}d\Sigma
_{p}+2\int_{M}Tor\left( (\nabla _{b}\varphi )_{\mathbb{C}},(\nabla
_{b}\varphi )_{\mathbb{C}}\right) d\mu .%
\end{array}
\label{100}
\end{equation}%
Also applying the divergence Theorem to the equation%
\begin{equation*}
\begin{array}{c}
(B^{\beta \overline{\gamma }}\varphi )(B_{\beta \overline{\gamma }}\varphi
)=(\varphi ^{\beta }B_{\beta \overline{\gamma }}\varphi ),^{\overline{\gamma 
}}-\frac{n-1}{n}(\varphi P_{\beta }\varphi ),^{\beta }+\frac{n-1}{8n}\varphi
P_{0}\varphi%
\end{array}%
\end{equation*}%
with $B_{\beta \overline{\gamma }}\varphi =\varphi _{\beta \overline{\gamma }%
}-\frac{1}{n}\varphi _{\sigma }{}^{\sigma }h_{\beta \overline{\gamma }},$ we
obtain%
\begin{equation}
\begin{array}{ll}
& \int_{M}\tsum_{\beta ,\gamma }|\varphi _{\beta \overline{\gamma }}-\frac{1%
}{n}\varphi _{\sigma }{}^{\sigma }h_{\beta \overline{\gamma }}|^{2}d\mu \\ 
= & \frac{n-1}{8n}\int_{M}\varphi P_{0}\varphi d\mu -\frac{n-1}{4n}%
iC_{n}\int_{\Sigma }\varphi \left( P_{n}\varphi -P_{\overline{n}}\varphi
\right) d\Sigma _{p} \\ 
& +\frac{1}{4}iC_{n}\int_{\Sigma }(\varphi ^{\overline{\beta }}B_{n\overline{%
\beta }}\varphi -\varphi ^{\beta }B_{\overline{n}\beta }\varphi )d\Sigma
_{p}.%
\end{array}
\label{101}
\end{equation}%
Here%
\begin{equation*}
\begin{array}{ll}
& i(\varphi ^{\overline{\beta }}B_{n\overline{\beta }}\varphi -\varphi
^{\beta }B_{\overline{n}\beta }\varphi ) \\ 
= & \frac{1}{4}\tsum_{\beta \neq n}[\varphi _{e_{n+\beta }}(\varphi
_{e_{\beta }e_{n}}+\varphi _{e_{n+\beta }e_{2n}})+\varphi _{e_{\beta
}}(\varphi _{e_{\beta }e_{2n}}-\varphi _{e_{n+\beta }e_{n}})] \\ 
& +\frac{1}{4}\varphi _{e_{2n}}[(\varphi _{e_{n}e_{n}}+\varphi
_{e_{2n}e_{2n}})-\frac{2}{n}\Delta _{b}\varphi ].%
\end{array}%
\end{equation*}%
Substituting these into the right hand side of (\ref{temporary05}), we get%
\begin{equation}
\begin{array}{ll}
& \frac{1}{2}\int_{M}\Delta _{b}|\nabla _{b}\varphi |^{2}d\mu \\ 
= & 2\int_{M}\tsum_{\beta ,\gamma }|\varphi _{\beta \gamma }|^{2}d\mu -\frac{%
n+1}{n}\int_{M}(\Delta _{b}\varphi )^{2}d\mu \\ 
& +\frac{n+2}{4n}\int_{M}\varphi P_{0}\varphi d\mu -\frac{n+2}{2n}%
iC_{n}\int_{\Sigma }\varphi \left( P_{n}\varphi -P_{\overline{n}}\varphi
\right) d\Sigma _{p} \\ 
& +\int_{M}[2Ric-(n+1)Tor]((\nabla _{b}\varphi )_{\mathbb{C}},(\nabla
_{b}\varphi )_{\mathbb{C}})d\mu +\frac{1}{4}C_{n}\int_{\Sigma }\varphi
_{0}\varphi _{e_{n}}d\Sigma _{p} \\ 
& +\frac{1}{2}iC_{n}\int_{\Sigma }(\varphi ^{\overline{\beta }}B_{n\overline{%
\beta }}\varphi -\varphi ^{\beta }B_{\overline{n}\beta }\varphi )d\Sigma
_{p}+\frac{2n+3}{4n}C_{n}\int_{\Sigma }(\Delta _{b}\varphi )\varphi
_{e_{2n}}d\Sigma _{p}.%
\end{array}
\label{temporary06}
\end{equation}

On the other hand, the divergence theorem (\ref{C}) implies that%
\begin{equation*}
\begin{array}{lll}
\frac{1}{2}\int_{M}\Delta _{b}|\nabla _{b}\varphi |^{2}d\mu & = & \frac{1}{4}%
C_{n}\int_{\Sigma }\left( |\nabla _{b}\varphi |^{2}\right) _{e_{2n}}d\Sigma
_{p} \\ 
& = & \frac{1}{4}C_{n}\int_{\Sigma }\tsum_{\beta \neq n}\left( \varphi
_{e_{\beta }}\varphi _{e_{\beta }e_{2n}}+\varphi _{e_{n+\beta }}\varphi
_{e_{n+\beta }e_{2n}}\right) d\Sigma _{p} \\ 
&  & +\frac{1}{4}C_{n}\int_{\Sigma }\left( \varphi _{e_{n}}\varphi
_{e_{n}e_{2n}}+\varphi _{e_{2n}}\varphi _{e_{2n}e_{2n}}\right) d\Sigma _{p}.%
\end{array}%
\end{equation*}%
Substituting the commutation relations 
\begin{equation*}
\begin{array}{l}
\varphi _{e_{\beta }e_{n+\gamma }}=\varphi _{e_{n+\gamma }e_{\beta }},\
\varphi _{e_{n+\beta }e_{n+\gamma }}=\varphi _{e_{n+\gamma }e_{n+\beta }},\ 
\mathrm{for}\text{ }\mathrm{all}\text{ }{\small \beta \neq \gamma ,} \\ 
\text{ }\varphi _{e_{n}e_{2n}}=\varphi _{e_{2n}e_{n}}+2\varphi _{0},%
\end{array}%
\end{equation*}%
and 
\begin{equation}
\begin{array}{c}
\tsum_{\beta \neq n}2(\varphi _{\beta \overline{\beta }}+\varphi _{\overline{%
\beta }\beta })+\varphi _{e_{n}e_{n}}=\sum_{j=1}^{2n-1}\varphi
_{e_{j}e_{j}}=2\Delta _{b}^{t}\varphi +H_{p.h}\varphi _{e_{2n}} \\ 
\varphi _{e_{2n}e_{2n}}=2\Delta _{b}\varphi -\sum_{j=1}^{2n-1}\varphi
_{e_{j}e_{j}}%
\end{array}
\label{A}
\end{equation}%
into the above equation, also integrating by parts from (\ref{2014a}) and (%
\ref{2014b}) yields 
\begin{equation}
\begin{array}{ll}
& \frac{1}{2}\int_{M}\Delta _{b}|\nabla _{b}\varphi |^{2}d\mu \\ 
= & \frac{1}{4}C_{n}\int_{\Sigma }\tsum_{\beta \neq n}(\varphi _{e_{\beta
}}\varphi _{e_{2n}e_{\beta }}+\varphi _{e_{n+\beta }}\varphi
_{_{e_{2n}e_{n+\beta }}})d\Sigma _{p} \\ 
& +\frac{1}{4}C_{n}\int_{\Sigma }\varphi _{e_{n}}(\varphi
_{e_{2n}e_{n}}+2\varphi _{0})d\Sigma _{p}+\frac{1}{4}C_{n}\int_{\Sigma
}\varphi _{e_{2n}}\varphi _{e_{2n}e_{2n}}d\Sigma _{p} \\ 
= & \frac{1}{4}C_{n}\int_{\Sigma }[\tsum_{\beta \neq n}2(\varphi _{\overline{%
\beta }}\varphi _{e_{2n}Z_{\beta }}+\varphi _{\beta }\varphi _{_{e_{2n}Z_{%
\overline{\beta }}}})+\varphi _{e_{n}}\varphi _{e_{2n}e_{n}}]d\Sigma _{p} \\ 
& +\frac{1}{2}C_{n}\int_{\Sigma }\varphi _{e_{n}}\varphi _{0}d\Sigma _{p}+%
\frac{1}{4}C_{n}\int_{\Sigma }\varphi _{e_{2n}}\varphi
_{e_{2n}e_{2n}}d\Sigma _{p} \\ 
= & -\frac{1}{4}C_{n}\int_{\Sigma }\varphi _{e_{2n}}[\tsum_{\beta \neq
n}2(\varphi _{\beta \overline{\beta }}+\varphi _{\overline{\beta }\beta
})+\varphi _{e_{n}e_{n}}]d\Sigma _{p} \\ 
& -\frac{1}{2}C_{n}\int_{\Sigma }\alpha \varphi _{e_{n}}\varphi
_{e_{2n}}d\Sigma _{p}-\frac{1}{4}C_{n}\int_{\Sigma }\varphi _{e_{n}}(\nabla
_{e_{n}}e_{2n})\varphi d\Sigma _{p} \\ 
& +\frac{1}{2}C_{n}\int_{\Sigma }\varphi _{e_{n}}\varphi _{0}d\Sigma _{p}+%
\frac{1}{4}C_{n}\int_{\Sigma }\varphi _{e_{2n}}\varphi
_{e_{2n}e_{2n}}d\Sigma _{p} \\ 
& +\frac{1}{4}C_{n}\int_{\Sigma }\varphi _{e_{2n}}[\tsum_{\beta \neq
n}(\theta _{n}{}^{\beta }(e_{n})\varphi _{\beta }+\theta _{\overline{n}}{}^{%
\overline{\beta }}(e_{n})\varphi _{\overline{\beta }})-(\nabla
_{e_{n}}e_{n})^{t}\varphi ]d\Sigma _{p} \\ 
& +\frac{1}{2}C_{n}\int_{\Sigma }\tsum_{\beta ,\gamma \neq n}i(\theta _{%
\overline{n}}{}^{\overline{\gamma }}(Z_{\beta })\varphi _{\overline{\gamma }%
}-\theta _{n}{}^{\gamma }(Z_{\beta })\varphi _{\gamma })\varphi _{\overline{%
\beta }}d\Sigma _{p} \\ 
& -\frac{1}{2}C_{n}\int_{\Sigma }\tsum_{\beta ,\gamma \neq n}i(\theta
_{n}{}^{\gamma }(Z_{\overline{\beta }})\varphi _{\gamma }-\theta _{\overline{%
n}}{}^{\overline{\gamma }}(Z_{\overline{\beta }})\varphi _{\gamma })\varphi
_{\beta }d\Sigma _{p} \\ 
= & \frac{1}{2}C_{n}\int_{\Sigma }\varphi _{e_{2n}}\left( \Delta _{b}\varphi
-2\Delta _{b}^{t}\varphi \right) d\Sigma _{p}-\frac{1}{4}C_{n}\int_{\Sigma
}H_{p.h}\varphi _{e_{2n}}^{2}d\Sigma _{p} \\ 
& -\frac{1}{2}C_{n}\int_{\Sigma }\alpha \varphi _{e_{n}}\varphi
_{e_{2n}}d\Sigma _{p}+\frac{1}{2}C_{n}\int_{\Sigma }\varphi _{0}\varphi
_{e_{n}}d\Sigma _{p} \\ 
& -\frac{1}{4}C_{n}\int_{\Sigma }\tsum_{j,k=1}^{2n-1}\left \langle \nabla
_{e_{j}}e_{2n},e_{k}\right \rangle \varphi _{e_{j}}\varphi _{e_{k}}d\Sigma
_{p}.%
\end{array}
\label{temporary09}
\end{equation}%
Here we use $\varphi _{\beta \overline{\beta }}=Z_{\overline{\beta }%
}Z_{\beta }\varphi -\tsum_{\gamma \neq n}\theta _{\beta }{}^{\gamma }(Z_{%
\overline{\beta }})\varphi _{\gamma }$ for each $\beta \neq n$ and $\varphi
_{e_{n}e_{n}}=e_{n}^{2}\varphi -(\nabla _{e_{n}}e_{n})^{t}\varphi $ on $%
\Sigma ,$ the fact that (\ref{A}) holds only on $\Sigma \backslash S_{\Sigma
}.$ However, $d\Sigma _{p}$ can be continuously extends over the singular
set $S_{\Sigma }$ and vanishes on $S_{\Sigma }.$ Finally, by combining the
equations (\ref{temporary06}) and (\ref{temporary09}), we can then obtain (%
\ref{CR Reilly's formula}). This completes the proof of Theorem.
\end{proof}

\section{The CR First Non-Zero Dirichlet Eigenvalue Estimate}

In this section, we derive the first Dirichlet eigenvalue estimate in a
compact pseudohermitian $(2n+1)$-manifold $(M,J,\theta )$ with boundary $%
\Sigma $.

\begin{lemma}
Let $(M,J,\theta )$ be a compact pseudohermitian $(2n+1)$-manifold with the
smooth boundary $\Sigma $ of pseudohermitian mean curvature $H_{p.h}$\ for $%
n\geq 2$. For the first eigenfunction $\varphi $ of Dirichlet eigenvalue
problem (\ref{1b}), we have%
\begin{equation}
\begin{array}{l}
\frac{n-1}{8n}\int_{M}\varphi P_{0}\varphi d\mu =\int_{M}\tsum_{\beta
,\gamma }|\varphi _{\beta \overline{\gamma }}-\frac{1}{n}\varphi _{\sigma
}{}^{\sigma }h_{\beta \overline{\gamma }}|^{2}d\mu +\frac{1}{8}%
C_{n}\int_{\Sigma }H_{p.h}\varphi _{e_{2n}}^{2}d\Sigma _{p}%
\end{array}
\label{2aa}
\end{equation}%
which implies 
\begin{equation}
\begin{array}{c}
\int_{M}\varphi P_{0}\varphi d\mu \geq 0%
\end{array}
\label{2bb}
\end{equation}%
if $H_{p.h}$ is nonnegative.
\end{lemma}

\begin{proof}
Since $\varphi =0$ on $\Sigma $ and $e_{j}$ is tangent along $\Sigma $ for $%
1\leq j\leq 2n-1$, then $\varphi _{e_{j}}=0$ for $1\leq j\leq 2n-1$ and $%
\Delta _{b}^{t}\varphi =\frac{1}{2}\sum_{j=1}^{2n-1}[e_{j}{}^{2}\varphi
-(\nabla _{e_{j}}e_{j})^{t}\varphi ]=0$ on $\Sigma .$ Furthermore, since $%
\Delta _{b}\varphi =\lambda _{1}\varphi \ $on$\mathrm{\ }M$ and $\varphi =0$
on $\Sigma ,$ then $\Delta _{b}\varphi =0$ on $\Sigma $.

It follows from (\ref{2014a}), (\ref{2014b}) and (\ref{A}) that 
\begin{equation}
\begin{array}{ll}
& iC_{n}\int_{\Sigma }(\varphi ^{\overline{\beta }}B_{n\overline{\beta }%
}\varphi -\varphi ^{\beta }B_{\overline{n}\beta }\varphi )d\Sigma _{p} \\ 
= & iC_{n}\int_{\Sigma }\sum_{\beta \neq n}(\varphi _{\beta }\varphi _{n%
\overline{\beta }}-\varphi _{\overline{\beta }}\varphi _{\overline{n}\beta
})d\Sigma _{p}+C_{n}\int_{\Sigma }(\varphi _{n}B_{n\overline{n}}\varphi
-\varphi _{\overline{n}}B_{\overline{n}n}\varphi )d\Sigma _{p} \\ 
= & iC_{n}\int_{\Sigma }\varphi _{n}[B_{n\overline{n}}\varphi -\sum_{\beta
\neq n}(\varphi _{\beta \overline{\beta }}-\frac{1}{2}\theta _{n}{}^{\beta
}(e_{n})\varphi _{\beta })]d\Sigma _{p}+\text{ }\mathrm{conjugate} \\ 
= & iC_{n}\int_{\Sigma }\varphi _{n}[2\varphi _{n\overline{n}}-\frac{n+1}{n}%
\varphi _{\gamma }{}^{\gamma }+\frac{1}{2}\sum_{\beta \neq n}\theta
_{n}{}^{\beta }(e_{n})\varphi _{\beta }]d\Sigma _{p}+\text{ }\mathrm{%
conjugate} \\ 
= & \frac{1}{2}C_{n}\int_{\Sigma }\varphi _{e_{2n}}[\varphi
_{e_{n}e_{n}}+(\nabla _{e_{n}}{}^{e_{n}})^{t}\varphi +\varphi
_{e_{2n}e_{2n}}-\frac{n+1}{n}\Delta _{b}\varphi ]d\Sigma _{p} \\ 
& +\frac{n-1}{2}C_{n}\int_{\Sigma }\varphi _{0}\varphi _{e_{n}}d\Sigma _{p}+%
\frac{1}{2}C_{n}\int_{\Sigma }\varphi _{e_{n}}[(\nabla
_{e_{n}}{}^{e_{2n}})^{t}\varphi +\tilde{\omega}_{n}^{\;n}(e_{n})\varphi
_{e_{n}}]d\Sigma _{p} \\ 
= & \frac{1}{2}C_{n}\int_{\Sigma }\varphi _{e_{2n}}(\frac{n-1}{n}\Delta
_{b}\varphi -2\Delta _{b}^{t}\varphi -H_{p.h}\varphi _{e_{2n}})d\Sigma _{p}+%
\frac{n-1}{2}C_{n}\int_{\Sigma }\varphi _{0}\varphi _{e_{n}}d\Sigma _{p} \\ 
& -\frac{1}{2}C_{n}\int_{\Sigma }\varphi _{e_{n}}\varphi
_{e_{2n}e_{n}}d\Sigma _{p}+\frac{1}{2}C_{n}\int_{\Sigma }\tilde{\omega}%
_{n}^{\;n}(e_{n})\varphi _{e_{n}}^{2}d\Sigma _{p}-C_{n}\int_{\Sigma }\alpha
\varphi _{e_{n}}\varphi _{e_{2n}}d\Sigma _{p},%
\end{array}
\label{2ab}
\end{equation}%
where we used $B_{n\overline{\beta }}\varphi =\varphi _{n\overline{\beta }}$
for $\beta \neq n,$ $B_{n\overline{n}}\varphi =\varphi _{n\overline{n}}-%
\frac{1}{n}\varphi _{\gamma }{}^{\gamma }$ and 
\begin{equation*}
\begin{array}{cl}
& \int_{\Sigma }[\varphi _{e_{2n}}(\varphi _{e_{n}e_{n}}+(\nabla
_{e_{n}}{}^{e_{n}})^{t}\varphi )+\varphi _{e_{n}}(\nabla
_{e_{n}}{}^{e_{2n}})^{t}\varphi ]d\Sigma _{p} \\ 
= & \int_{\Sigma }[\varphi _{e_{2n}}(e_{n})^{2}\varphi +\varphi
_{e_{n}}(\nabla _{e_{n}}{}^{e_{2n}})^{t}\varphi ]d\Sigma _{p} \\ 
= & -\int_{\Sigma }[\varphi _{e_{n}}(e_{n}e_{2n}\varphi -(\nabla
_{e_{n}}{}^{e_{2n}})^{t}\varphi )+2\alpha \varphi _{e_{n}}\varphi
_{e_{2n}}]d\Sigma _{p} \\ 
= & -\int_{\Sigma }(\varphi _{e_{n}}\varphi _{e_{2n}e_{n}}+2\alpha \varphi
_{e_{n}}\varphi _{e_{2n}})d\Sigma _{p}.%
\end{array}%
\end{equation*}%
Substituting (\ref{2ab}) into (\ref{101}), we get%
\begin{equation*}
\begin{array}{lll}
\frac{n-1}{8n}\int_{M}\varphi P_{0}\varphi d\mu & = & \int_{M}\tsum_{\beta
,\gamma }|\varphi _{\beta \overline{\gamma }}-\frac{1}{n}\varphi _{\sigma
}{}^{\sigma }h_{\beta \overline{\gamma }}|^{2}d\mu +\frac{n-1}{4n}%
iC_{n}\int_{\Sigma }\varphi \left( P_{n}\varphi -P_{\overline{n}}\varphi
\right) d\Sigma _{p} \\ 
&  & -\frac{1}{4}iC_{n}\int_{\Sigma }(\varphi ^{\overline{\beta }}B_{n%
\overline{\beta }}\varphi -\varphi ^{\beta }B_{\overline{n}\beta }\varphi
)d\Sigma _{p} \\ 
& = & \int_{M}\tsum_{\beta ,\gamma }|\varphi _{\beta \overline{\gamma }}-%
\frac{1}{n}\varphi _{\sigma }{}^{\sigma }h_{\beta \overline{\gamma }%
}|^{2}d\mu +\frac{n-1}{4n}iC_{n}\int_{\Sigma }\varphi \left( P_{n}\varphi
-P_{\overline{n}}\varphi \right) d\Sigma _{p} \\ 
&  & -\frac{1}{8}C_{n}\int_{\Sigma }\varphi _{e_{2n}}(\frac{n-1}{n}\Delta
_{b}\varphi -2\Delta _{b}^{t}\varphi -H_{p.h}\varphi _{e_{2n}})d\Sigma _{p}-%
\frac{n-1}{8}C_{n}\int_{\Sigma }\varphi _{0}\varphi _{e_{n}}d\Sigma _{p} \\ 
&  & +\frac{1}{8}C_{n}\int_{\Sigma }\varphi _{e_{n}}\varphi
_{e_{2n}e_{n}}d\Sigma _{p}-\frac{1}{8}C_{n}\int_{\Sigma }\tilde{\omega}%
_{n}^{\;n}(e_{n})\varphi _{e_{n}}^{2}d\Sigma _{p}+\frac{1}{4}%
C_{n}\int_{\Sigma }\alpha \varphi _{e_{n}}\varphi _{e_{2n}}d\Sigma _{p}%
\end{array}%
\end{equation*}%
which is the equation (\ref{2aa}) under the assumptions.
\end{proof}

Now we are ready to prove Theorem \ref{TB}.

\textbf{The Proof of Theorem \ref{TB}}:

\begin{proof}
It follows the CR Reilly formula (\ref{CR Reilly's formula}) that%
\begin{equation}
\begin{array}{c}
\frac{n+1}{n}\int_{M}(\Delta _{b}\varphi )^{2}d\mu \geq \frac{n+2}{4n}%
\int_{M}\varphi P_{0}\varphi d\mu +\int_{M}[2Ric-(n+1)Tor]((\nabla
_{b}\varphi )_{\mathbb{C}},(\nabla _{b}\varphi )_{\mathbb{C}})d\mu .%
\end{array}
\label{2a}
\end{equation}%
Since 
\begin{equation*}
\varphi =0\  \mathrm{and}\  \Delta _{b}\varphi =0\  \mathrm{on}\  \Sigma ,
\end{equation*}%
(\ref{1a}) and (\ref{2a}) imply 
\begin{equation*}
\begin{array}{l}
\frac{n+1}{n}\int_{M}(\Delta _{b}\varphi )^{2}d\mu \geq \frac{n+2}{4n}\mu
_{_{D}}^{1}\int_{M}\varphi ^{2}d\mu +\int_{M}[2Ric-(n+1)Tor]((\nabla
_{b}\varphi )_{\mathbb{C}},(\nabla _{b}\varphi )_{\mathbb{C}})d\mu .%
\end{array}%
\end{equation*}%
Moreover, by using 
\begin{equation*}
\begin{array}{c}
\lbrack 2Ric-(n+1)Tor]((\nabla _{b}\varphi )_{\mathbb{C}},(\nabla
_{b}\varphi )_{\mathbb{C}})\geq k|\nabla _{b}\varphi |^{2}%
\end{array}%
\end{equation*}%
and 
\begin{equation*}
\begin{array}{c}
\int_{M}|\nabla _{b}\varphi |^{2}d\mu =\lambda _{1}\int_{M}\varphi ^{2}d\mu ,%
\end{array}%
\end{equation*}%
we obtain%
\begin{equation*}
\begin{array}{l}
\frac{n+1}{n}\lambda _{1}^{2}\int_{M}\varphi ^{2}d\mu \geq (k\lambda _{1}+%
\frac{n+2}{4n}\mu _{_{D}}^{1})\int_{M}\varphi ^{2}d\mu .%
\end{array}%
\end{equation*}%
Hence 
\begin{equation*}
\begin{array}{c}
\frac{n+1}{n}\lambda _{1}^{2}-k\lambda _{1}-\frac{n+2}{4n}\mu
_{_{D}}^{1}\geq 0%
\end{array}%
\end{equation*}%
and thus%
\begin{equation*}
\begin{array}{c}
\lambda _{1}\geq \frac{nk+\sqrt{n^{2}k^{2}+(n+1)(n+2)\mu _{_{D}}^{1}}}{2(n+1)%
}.%
\end{array}%
\end{equation*}

(i) In case for $n=1,$ we have%
\begin{equation*}
\begin{array}{c}
\lambda _{1}\geq \frac{k+\sqrt{k^{2}+6\mu _{_{D}}^{1}}}{4},%
\end{array}%
\end{equation*}%
for $\mu _{_{D}}^{1}\geq -\frac{k^{2}}{6}$. In addition if $P_{0}$ is
nonnegative, we have 
\begin{equation*}
\begin{array}{c}
\lambda _{1}\geq \frac{k}{2}.%
\end{array}%
\end{equation*}

(i) In case for $n\geq 2,$ it follows from (\ref{2bb}) and (\ref{2a}) that 
\begin{equation*}
\frac{n+1}{n}\lambda _{1}^{2}-k\lambda _{1}\geq 0
\end{equation*}%
and then 
\begin{equation*}
\begin{array}{c}
\lambda _{1}\geq \frac{nk}{(n+1)}.%
\end{array}%
\end{equation*}
\end{proof}

\section{The First Eigenvalue Estimate of Embedded $P$-minimal hypersurfaces}

In this section, we study a CR analogue of Yau conjecture \cite{Y} on the
first eigenvalue estimate of embedded $p$-minimal hypersurfaces.

\textbf{The Proof of Theorem}\textup{\textbf{\  \ref{TC}:}}

\begin{proof}
Since $M$ has vanishing torsion and positive pseudohermitian Ricci
curvature, it follows from \cite{cc1} that $M$ has positive Ricci curvature
with respect to the Webster metric. Hence its first homology group $H^{1}(M,%
\mathbb{R})$ is trivial. By an exact sequence argument, we conclude that $%
\Sigma $ divides $M$ into two connected components $M_{1}$ and $M_{2}$ with $%
\partial M_{1}=\Sigma =\partial M_{2}$. Let us denote $D$ to be one of two
components to be chosen later. If $u$ is the first nonconstant eigenfunction
on $\Sigma $, satisfying 
\begin{equation*}
L_{\alpha }u=-\lambda _{1}u.
\end{equation*}%
We first let $\varphi $ be the solution of 
\begin{equation*}
\Delta _{b}\varphi =0\text{\  \textrm{on}\ }D
\end{equation*}%
with the boundary condition 
\begin{equation*}
\varphi =u\text{\  \textrm{on}\ }\Sigma .
\end{equation*}

If $D$ is a compact pseudohermitian $(2n+1)$-manifold with the smooth
boundary $\Sigma ,$ then $P_{0}$ is self-adjoint on the space of all smooth
functions with $\Delta _{b}\varphi =0$ and $(\Delta _{b}\varphi )_{e_{2n}}=0$
on $\Sigma $. In fact, it suffices to check that 
\begin{equation}
\begin{array}{ccl}
\int_{D}g\Delta _{b}^{2}fd\mu & = & -\int_{D}\left \langle \nabla
_{b}g,\nabla _{b}\Delta _{b}f\right \rangle d\mu +C_{n}\int_{\Sigma
}g(\Delta _{b}f)_{e_{2n}}d\Sigma _{p} \\ 
& = & \int_{D}\Delta _{b}f\Delta _{b}gd\mu -C_{n}\int_{\Sigma
}g_{e_{2n}}\Delta _{b}fd\Sigma _{p}+C_{n}\int_{\Sigma }g(\Delta
_{b}f)_{e_{2n}}d\Sigma _{p} \\ 
& = & \int_{D}\Delta _{b}f\Delta _{b}gd\mu =\int_{D}f\Delta _{b}^{2}gd\mu%
\end{array}
\label{A1}
\end{equation}%
and for $\alpha =0$%
\begin{equation}
\begin{array}{ccl}
\int_{D}gf_{00}d\mu & = & -\int_{D}g_{0}f_{0}d\mu +2C_{n}\int_{\Sigma
}\alpha gf_{0}d\Sigma _{p} \\ 
& = & \int_{D}fg_{00}d\mu -2C_{n}\int_{\Sigma }\alpha fg_{0}d\Sigma
_{p}+2C_{n}\int_{\Sigma }\alpha gf_{0}d\Sigma _{p} \\ 
& = & \int_{D}fg_{00}d\mu .%
\end{array}
\label{A2}
\end{equation}%
It follows that if the torsion is vanishing 
\begin{equation}
\begin{array}{c}
\int_{D}\varphi P_{0}\varphi d\mu \geq 0.%
\end{array}
\label{aaa}
\end{equation}%
By applying the CR Reilly formula (\ref{Reilly'sformula}), we have 
\begin{equation}
\begin{array}{lll}
0 & \geq & k\int_{D}|\nabla _{b}\varphi |^{2}d\mu -\frac{1}{4}%
C_{n}\int_{\Sigma }\varphi _{0}\varphi _{e_{n}}d\Sigma _{p}-\frac{n+2}{2n}%
iC_{n}\int_{\Sigma }\varphi \left( P_{n}\varphi -P_{\overline{n}}\varphi
\right) d\Sigma _{p} \\ 
&  & +\frac{i}{2}C_{n}\int_{\Sigma }(\varphi ^{\overline{\beta }}B_{n%
\overline{\beta }}\varphi -\varphi ^{\beta }B_{\overline{n}\beta }\varphi
)d\Sigma _{p}+\frac{3}{4n}C_{n}\int_{\Sigma }\varphi _{e_{2n}}\Delta
_{b}\varphi d\Sigma _{p} \\ 
&  & +C_{n}\int_{\Sigma }\varphi _{e_{2n}}\Delta _{b}^{t}\varphi d\Sigma
_{p}+\frac{1}{2}C_{n}\int_{\Sigma }\alpha \varphi _{e_{n}}\varphi
_{e_{2n}}d\Sigma _{p} \\ 
&  & +\frac{1}{4}C_{n}\int_{\Sigma }\left \langle \nabla
_{e_{i}}e_{2n},e_{j}\right \rangle \varphi _{e_{i}}\varphi _{e_{j}}d\Sigma
_{p}.%
\end{array}
\label{00}
\end{equation}%
Now we are going to estimate all terms in RHS of (\ref{00}):

(i) By the CR divergence theorem and $\Delta _{b}\varphi ^{2}=2\varphi
\Delta _{b}\varphi +2|\nabla _{b}\varphi |^{2}=2|\nabla _{b}\varphi |^{2},$
we have 
\begin{equation}
\begin{array}{cl}
& C_{n}\int_{\Sigma }\varphi _{e_{2n}}\Delta _{b}^{t}\varphi d\Sigma _{p} \\ 
= & -C_{n}\int_{\Sigma }\alpha \varphi _{e_{n}}\varphi _{e_{2n}}d\Sigma
_{p}-\lambda _{1}C_{n}\int_{\Sigma }\varphi \varphi _{e_{2n}}d\Sigma _{p} \\ 
= & -C_{n}\int_{\Sigma }\alpha \varphi _{e_{n}}\varphi _{e_{2n}}d\Sigma _{p}-%
\frac{1}{2}\lambda _{1}C_{n}\int_{\Sigma }(\varphi ^{2})_{e_{2n}}d\Sigma _{p}
\\ 
= & -C_{n}\int_{\Sigma }\alpha \varphi _{e_{n}}\varphi _{e_{2n}}d\Sigma
_{p}-\lambda _{1}\int_{D}(\Delta _{b}\varphi ^{2})d\mu \\ 
= & -C_{n}\int_{\Sigma }\alpha \varphi _{e_{n}}\varphi _{e_{2n}}d\Sigma
_{p}-2\lambda _{1}\int_{D}|\nabla _{b}\varphi |^{2}d\mu .%
\end{array}
\label{01}
\end{equation}

(ii) By the CR Green theorem 
\begin{equation}
\begin{array}{c}
\frac{3}{4n}C_{n}\int_{\Sigma }\varphi _{e_{2n}}\Delta _{b}\varphi d\Sigma
_{p}=\frac{3}{4n}\int_{D}(\Delta _{b}\varphi )^{2}d\mu +\frac{3}{4n}%
\int_{D}\left \langle \nabla _{b}\Delta _{b}\varphi ,\nabla _{b}\varphi
\right \rangle d\mu =0.%
\end{array}
\label{02}
\end{equation}

(iii) The same computation as (\ref{2ab}) for $\alpha =0$ and from (\ref{A}) 
\begin{equation}
\begin{array}{ll}
& \frac{i}{2}C_{n}\int_{\Sigma }(\varphi ^{\overline{\beta }}B_{n\overline{%
\beta }}\varphi -\varphi ^{\beta }B_{\overline{n}\beta }\varphi )d\Sigma _{p}
\\ 
= & \frac{1}{4}C_{n}\int_{\Sigma }\varphi _{e_{2n}}(\varphi _{e_{n}e_{n}}+%
\frac{n-1}{n}\Delta _{b}\varphi -2\Delta _{b}^{t}\varphi -H_{p.h}\varphi
_{e_{2n}})d\Sigma _{p}+\frac{n-1}{4}C_{n}\int_{\Sigma }\varphi _{0}\varphi
_{e_{n}}d\Sigma _{p} \\ 
= & \frac{1}{4}C_{n}\int_{\Sigma }\varphi _{e_{2n}}(\frac{n-1}{n}\Delta
_{b}\varphi -\sum_{j\neq n,2n}\varphi _{e_{j}e_{j}})d\Sigma _{p}+\frac{n-1}{4%
}C_{n}\int_{\Sigma }\varphi _{0}\varphi _{e_{n}}d\Sigma _{p}.%
\end{array}
\label{03}
\end{equation}

(iv) By straightforward computation, since $A_{\beta \gamma }=0$ 
\begin{equation*}
\begin{array}{c}
i\left( P_{n}\varphi -P_{\overline{n}}\varphi \right) =i\left( \varphi _{%
\overline{\beta }}{}^{\overline{\beta }}{}_{n}-\varphi _{\beta }{}^{\beta
}{}_{\overline{n}}\right) =\frac{1}{2}[n\varphi _{0e_{n}}+(\Delta
_{b}\varphi )_{e_{2n}}].%
\end{array}%
\end{equation*}%
From (\ref{2014b}), (\ref{03}) and $\int_{\Sigma }\varphi (\Delta
_{b}\varphi )_{e_{2n}}d\Sigma _{p}=0$ that 
\begin{equation}
\begin{array}{ccl}
-\frac{n+2}{2n}iC_{n}\int_{\Sigma }\varphi (P_{n}\varphi -P_{\overline{n}%
}\varphi )d\Sigma _{p} & = & -\frac{n+2}{4n}C_{n}\int_{\Sigma }\varphi
\lbrack n\varphi _{0e_{n}}+(\Delta _{b}\varphi )_{e_{2n}}]d\Sigma _{p} \\ 
& = & \frac{n+2}{4}C_{n}\int_{\Sigma }\varphi _{0}\varphi _{e_{n}}d\Sigma
_{p}+\frac{n+2}{2}C_{n}\int_{\Sigma }\alpha \varphi _{0}\varphi d\Sigma _{p}.%
\end{array}
\label{04}
\end{equation}%
By combining (\ref{00}), (\ref{01}), (\ref{02}), (\ref{03}) and (\ref{04})
for $\alpha =0$%
\begin{equation*}
\begin{array}{lll}
0 & \geq & (k-2\lambda _{1})\int_{D}|\nabla _{b}\varphi |^{2}d\mu +\frac{n}{2%
}C_{n}\int_{\Sigma }\varphi _{0}\varphi _{e_{n}}d\Sigma _{p} \\ 
&  & -\frac{1}{4}C_{n}\int_{\Sigma }\sum_{j\neq n,2n}\varphi
_{e_{j}e_{j}}\varphi _{e_{2n}}d\Sigma _{p}+\frac{1}{4}C_{n}\int_{\Sigma
}\left \langle \nabla _{e_{i}}e_{2n},e_{j}\right \rangle \varphi
_{e_{i}}\varphi _{e_{j}}d\Sigma _{p}.%
\end{array}%
\end{equation*}%
Moreover if $\alpha =0$, then the $p$-area element $d\Sigma _{p}$ is the
area form $d\Sigma $ on $\Sigma $ and

\begin{equation}
\begin{array}{lll}
0 & \geq & (k-2\lambda _{1})\int_{D}|\nabla _{b}\varphi |^{2}d\mu +\frac{n}{2%
}C_{n}\int_{\Sigma }\varphi _{0}\varphi _{e_{n}}d\Sigma \\ 
&  & -\frac{1}{4}C_{n}\int_{\Sigma }\sum_{j\neq n,2n}\varphi
_{e_{j}e_{j}}\varphi _{e_{2n}}d\Sigma +\frac{1}{4}C_{n}\int_{\Sigma
}\left \langle \nabla _{e_{i}}e_{2n},e_{j}\right \rangle \varphi
_{e_{i}}\varphi _{e_{j}}d\Sigma .%
\end{array}
\label{05}
\end{equation}

Next we observe that $T$ is always tangent to $\Sigma $ due to $\alpha =0$.
Then $\int_{\Sigma }\varphi _{0}\varphi _{e_{n}}d\Sigma $ is independent of
the extended function $\varphi $. If we choose a different component of $%
M\backslash \Sigma $ to perform this computation, $u_{e_{n}}u_{0},$ $%
\sum_{j\neq n,2n}u_{e_{j}e_{j}}u_{e_{2n}}$ and $\left \langle \nabla
_{e_{i}}e_{2n},e_{j}\right \rangle u_{e_{i}}u_{e_{j}}$ will differ by a sign,
hence we may choose a component, say $M_{1}$, so that 
\begin{equation}
\begin{array}{c}
2n\int_{\Sigma }\varphi _{0}\varphi _{e_{n}}d\Sigma -\int_{\Sigma
}\sum_{j\neq n,2n}\varphi _{e_{j}e_{j}}\varphi _{e_{2n}}d\Sigma
+\int_{\Sigma }\left \langle \nabla _{e_{i}}e_{2n},e_{j}\right \rangle \varphi
_{e_{i}}\varphi _{e_{j}}d\Sigma \geq 0.%
\end{array}
\label{06}
\end{equation}%
By combining (\ref{05}) and (\ref{06}) that we have 
\begin{equation*}
\begin{array}{c}
0\geq (k-2\lambda _{1})\int_{D}|\nabla _{b}\varphi |^{2}d\mu%
\end{array}%
\end{equation*}%
with $D=M_{1}.$ This implies 
\begin{equation*}
0\geq k-2\lambda _{1}
\end{equation*}%
and thus%
\begin{equation*}
\begin{array}{c}
\lambda _{1}\geq \frac{k}{2}%
\end{array}%
\end{equation*}%
because $\varphi $ has boundary value $u$ which is nonconstant.

Now if the equality holds for $n=1.$, then 
\begin{equation*}
W=k.
\end{equation*}%
Since $A_{11}=0$, 
\begin{equation*}
Q_{11}=0
\end{equation*}%
and then $(M,J,\theta )$ is a closed spherical pseudohermitian $3$-manifold.
On the other hand, it follows from (\cite{chmy}) that any embedded $p$%
-minimal surface in a closed spherical pseudohermitian $3$-manifold must
have genus less than two. In additional, if $M$ is simply connected, then $%
(M,J,\theta )$ is the standard pseudohermitian $3$-sphere. This completes
the proof.
\end{proof}

\end{document}